\documentclass[12pt]{amsart}
\usepackage{amssymb,amscd}

\swapnumbers 
\newtheorem{thm}[subsection]{Theorem}
\newtheorem{lem}[subsection]{Lemma}
\newtheorem{prop}[subsection]{Proposition}
\newtheorem{cor}[subsection]{Corollary}
\newtheorem*{conj*}{Conjecture}
\theoremstyle{definition}

\newtheorem*{example*}{Example}

\theoremstyle{remark}
\newtheorem{remark}[subsection]{Remark}
\numberwithin{equation}{subsection}

\newcommand{\N}{{\mathbb N}}
\newcommand{\Z}{{\mathbb Z}}
\newcommand{\Q}{{\mathbb Q}}

\newcommand{\End}{\operatorname{End}}

\newcommand{\Hom}{\operatorname{Hom}}

\newcommand{\SL}{\mathsf{SL}}
\newcommand{\gl}{\mathfrak{gl}}

\newcommand{\so}{\mathfrak{so}}
\renewcommand{\sp}{\mathfrak{sp}}
\newcommand{\g}{\mathfrak{g}}
\newcommand{\h}{\mathfrak{h}}
\newcommand{\divided}[2]{#1^{(#2)}}
\newcommand{\sqbinom}[2]{\begin{bmatrix}#1\\#2\end{bmatrix}}
\newcommand{\UU}{\mathbf{U}}
\newcommand{\Sch}{\mathbf{S}}
\newcommand{\SchP}{\mathbf{\widetilde{S}}}
\newcommand{\T}{\mathbf{T}}

\newcommand{\A}{\mathcal{A}}

\newcommand{\bil}[2]{\langle #1, #2 \rangle}

\newcommand{\tilK}{\widetilde{K}}

\newcommand{\lp}{\lambda^\prime}
\newcommand{\lpp}{\lambda^{\prime\prime}}

\newcommand{\al}{\alpha}

\newcommand{\ideal}{\Lambda}
\newcommand{\F}{\mathbb{F}}
\newcommand{\Aff}{\mathbb{A}}
\newcommand{\U}{\mathfrak{U}}
\renewcommand{\ge}{\geqslant}
\renewcommand{\le}{\leqslant}
\newcommand{\pt}[1]{\item{\rm(#1)}}

\begin{document}
\title[Generalized $q$-Schur algebras]
{On the defining relations for \\
generalized $q$-Schur algebras}

\author[Doty]{Stephen Doty}
\address{Department of Mathematics and Statistics\\
Loyola University Chi\-cago\\
Chicago, Illinois 60626 USA}
\email{doty@math.luc.edu}

\author[Giaquinto]{Anthony Giaquinto}
\address{Department of Mathematics and Statistics\\
Loyola University Chi\-cago\\
Chicago, Illinois 60626 USA}
\email{tonyg@math.luc.edu}

\author[Sullivan]{John Sullivan}
\address{Department of Mathematics, University of Washington, Seattle,
Washington 98195-4350 USA}
\email{sullivan@math.washington.edu}

\date{Revised 27 April 2008}

\begin{abstract}
We show that the defining relations needed to describe a generalized
$q$-Schur algebra as a quotient of a quantized enveloping algebra are
determined completely by the defining ideal of a certain finite affine
variety, the points of which correspond bijectively to the set of
weights. This explains, unifies, and extends previous results.
\end{abstract}
\maketitle

\parskip=2pt
\allowdisplaybreaks

\section*{Introduction}\noindent
Consider a reductive Lie algebra $\g$ and an associated quantized
enveloping algebra $\UU$ (over $\Q(v)$, $v$ an indeterminate)
determined by a given root datum $(X,\{\alpha_i\}, Y, \{h_i\})$ of
finite type. (See \S\ref{sec:notation} for basic notation and standard
terminology.) In some sense $\UU$ is a deformation of the universal
enveloping algebra $\U = \U(\g)$ (taken over $\Q$).  A generalized
$q$-Schur algebra $\Sch(\pi)$ determined by a finite saturated set
$\pi \subset X^+$ is the quotient $\UU/\Lambda$ where $\Lambda$
consists of all $u \in \UU$ acting as zero on all finite dimensional
$\UU$-modules admitting a weight space decomposition indexed by $W\pi$
($W$ is the Weyl group). Thus one truncates the category of
$\UU$-modules to the full subcategory $\mathcal{C}[\pi]$ consisting of
those modules admitting such a weight space decomposition, and the
module category of the finite dimensional algebra $\Sch(\pi)$ is
precisely the category $\mathcal{C}[\pi]$.

We consider the general problem of describing in some explicit way the
extra relations needed to define $\Sch(\pi)$ as a quotient of $\UU$;
that is, we seek to find a reasonable set of generators of the ideal
$\Lambda$ in terms of the usual generators of $\UU$.  Then $\Sch(\pi)$
is the algebra given by the usual generators of $\UU$ with its usual
defining relations along with the extra relations generating
$\Lambda$.

We also consider the analogous question in the limit at $v=1$, where
$\Sch(\pi)$ is replaced by a generalized Schur algebra $S(\pi)$ and
$\UU$ is replaced by $\U$. The generalized Schur algebra $S(\pi)$ was
introduced by Donkin \cite{SA1} as the quotient $\U/\Lambda$ where
again $\Lambda$ consists of all $u \in \U$ which act as zero on all
finite dimensional $\U$-modules admitting a weight space decomposition
indexed by the set $W\pi$.

Our main results are Theorems \ref{thm:main} and
\ref{thm:classical:main} and Corollaries \ref{cor:main} and
\ref{cor:classical:main}. The theorems state that the extra relations
may be found simply by locating a set of generators for the vanishing
ideal of a certain discrete affine variety corresponding to the set
$W\pi$ of weights ($W$ is the Weyl group), and the corollaries give
explicit formulas for a set of generators of that ideal. More
precisely, we choose an arbitrary $\Z$-basis $H_1, \dots, H_n$ of the
abelian group $Y$. In the classical case, the discrete set of points
is
\[
\{ (\bil{H_1}{\lambda}, \dots, \bil{H_n}{\lambda})
\mid \lambda \in W\pi \}
\]
and the vanishing ideal is the ideal of all elements of the polynomial
algebra $\Q[H_1, \dots, H_n] \simeq \U^0$ vanishing on this set of
points in $\Aff_\Q^n$, where we regard the $H_a$ for $a=1,\dots, n$ as
coordinate functions on $\Aff_\Q^n$. Unlike the quantized case, here
we can take $H_1, \dots, H_n$ to be any $\Q$-basis of $\h = \Q
\otimes_\Z Y \subset \g$; see \ref{def:Ualt}. This is a major
difference between the generalized Schur algebras and their
$q$-analogues. 

In the general (quantized) case, the discrete set of points is
\[
\{ (v^{\bil{H_1}{\lambda}}, \dots,
v^{\bil{H_n}{\lambda}}) \mid \lambda \in W\pi \}
\]
and the vanishing ideal is the ideal of all elements of the polynomial
algebra $\Q[K_{H_1}, \dots, K_{H_n}] \subset \UU^0$ vanishing on this
set of points in $\Aff_{\Q(v)}^n$, where we regard the $K_{H_a}$ for
$a=1,\dots, n$ as coordinate functions on $\Aff_\Q^n$.  Then the
defining ideal of the generalized Schur algebra determined by $\pi$ is
the ideal of $\U$ (resp., $\UU$) generated by the vanishing ideal.

Note that it immediately follows from these results that the extra
defining relations needed to define a generalized Schur algebra as a
quotient of $\UU$ (resp., $\U$) may be found in the zero part of $\UU$
(resp., $\U$).

The compelling generality and simplicity of this description lies in
the wealth of examples that are brought together under its
jurisdiction. We easily recover previous results from \cite{DG:PSA},
\cite{DD}, \cite{DGS} as special cases of the main result; see
\S\ref{sec:ex}. The reader is advised to start with the main results,
followed by the examples, before tracing through the proofs of the
main results, which are somewhat technical.

In order to prove our results, we rely on the presentation (in terms
of idempotents in place of Cartan generators) of generalized Schur
algebras given in \cite{PGSA}. It was necessary to extend the results
of \cite{PGSA} to our slightly more general setup, the details of
which are contained in Sections \ref{app:A} and \ref{app:B}.

One motivation for these results lies in a desire to investigate new
instances of Schur-Weyl duality. Let $M$ be a finite dimensional
module for $\UU$ (or $\U$) which admits a weight space decomposition;
i.e., an object of $\mathcal{C}$ (see \ref{modules}). Write $\Pi^+(M)$
for the set of dominant weights of $M$. The set $\pi = \Pi^+(M)$ is
necessarily saturated, in the sense of \ref{gsa}, so it determines a
generalized Schur algebra $\Sch(\pi)$. We say that the {\em module}
$M$ is saturated if the set of highest weights of its composition
factors coincides with $\Pi^+(M)$.  If $M$ is a saturated module, then
we may identify the generalized Schur algebra $\Sch(\pi)$ with the
image $\rho(\UU)$ of the representation $\rho: \UU \to \End(M)$
affording the $\UU$-module structure on $M$. This condition is
necessary for the functor $\Hom_{\Sch(\pi)}(M,-)$ to provide a strong
connection between the representation theories of $\UU$ and its
centralizer algebra $\End_{\UU}(M) = \End_{\Sch(\pi)}(M)$.  

In type $A$ if one takes $M$ to be a tensor power of the natural
module, then $\Sch(\pi)$ is the $q$-Schur algebra and $S(\pi)$ the
classical Schur algebra, see \ref{ex:A:classical}.  In type $B$ the
tensor powers of the natural module are not generally saturated
modules, so in \ref{ex:B_n:Spin} we treat the generalized Schur
algebra coming from tensor powers of the spin representation, which
are in fact always saturated modules (see Appendix~\ref{app:spin}).

It is worth mentioning here an interesting conjecture relating to the
question of saturation of tensor powers.

\begin{conj*}
  Let $V$ be a finite dimensional irreducible $\UU$-module which is an
  object of $\mathcal{C}$ (see \ref{modules}).  Then the module
  $V^{\otimes r}$ is saturated for all $r\ge 0$ if and only if $V$ is
  minuscule.
\end{conj*}

\noindent
There is some evidence for this conjecture. It is easy to see that in
order for $V^{\otimes r}$ to be saturated for all $r$ it is necessary
that $V$ be minuscule. Taking $r=1$ we see that $V$ itself must be
saturated, which forces it to be minuscule. Moreover, the converse is
known to hold in case $V$ is the natural module in types $A$, $C$,
$D$, as well as when $V$ is the spin module in type $B$. (These are
all minuscule.)

\section{Basic Notation}\label{sec:notation}\noindent
Notation introduced here will be used throughout.  Our conventions are
similar to those of Lusztig's book \cite{Lusztig}.

\subsection{Cartan datum}
Let a Cartan datum be given. By definition, a Cartan datum consists of
a finite set $I$ and a symmetric bilinear form $(\ ,\ )$ on the free
abelian group $\Z[I]$ taking values in $\Z$, such that:

(a) $(i,i) \in \{2,4,6,\dots \}$ for any $i$ in $I$.

(b) $2(i,j)/(i,i) \in \{0, -1, -2, \dots \}$ for any $i \ne j$ in $I$.

\subsection{Weyl group}
The Weyl group $W$ associated to the given Cartan datum is defined as
follows. For any $i \ne j$ in $I$ such that $(i,i)(j,j)-(i,j)^2 > 0$
let $h(i,j) \in \{2,3,4,\dots \}$ be determined by the equality
\[
\cos^2 \frac{\pi}{h(i,j)} = \frac{(i,j)}{(i,i)}\,\frac{(j,i)}{(j,j)}.
\]
One has $h(i,j) = h(j,i) = 2,3,4$ or $6$ according to whether
$\frac{2(i,j)}{(i,i)}\, \frac{2(j,i)}{(j,j)}$ is 0, 1, 2, or 3. For
any $i \ne j$ such that $(i,i)(j,j)-(i,j)^2 \le 0$ one sets $h(i,j) =
\infty$.

The braid group is the group given by generators $s_i$ ($i \in I$)
satisfying for each $i\ne j$ such that $h(i,j) < \infty$ the relations
\[
 s_i s_j \cdots = s_j s_i \cdots
\]
where on both sides of the equality we alternate between the factors
$s_i$ and $s_j$ in the given order until there are $h(i,j)$ factors on
each side. The Weyl group $W$ is the group on the same set of
generators, satisfying the same relations, along with the relations
$s_i^2 = 1$ for all $i \in I$. It is naturally a quotient of the braid
group.

\subsection{Root datum}\label{root}
A root datum associated to the given Cartan datum consists of two
finitely generated free abelian groups $X$, $Y$ and a perfect bilinear
pairing $\bil{\ }{\ }: Y \times X \to \Z$ along with embeddings $I \to
Y$ ($i \mapsto h_i$) and $I \to X$ ($i \mapsto \alpha_i$) such that
\[
   \bil{h_i}{\alpha_j} = 2\frac{(i,j)}{(i,i)}
\]
for all $i,j$ in $I$.

The image of the embedding $I \to Y$ is the set $\{h_i\}$ of simple
coroots and the image of the embedding $I \to X$ is the set
$\{\alpha_i\}$ of simple roots.

\subsection{}
The assumptions on the root datum imply that

(a) $\bil{h_i}{\alpha_i} = 2$ for all $i\in I$;

(b) $\bil{h_i}{\alpha_j} \in \{0, -1, -2, \dots \}$ for all $i \ne j
\in I$.

\noindent
In other words, the matrix $(\, \bil{h_i}{\alpha_j} \,)$ indexed by $I
\times I$ is a symmetrizable generalized Cartan matrix.

For each $i \in I$ we set $d_i = (i,i)/2$ (note that $d_i \in
\{1,2,3\}$). Then the matrix $(\,d_i \bil{h_i}{\alpha_j} \,)$ indexed
by $I \times I$ is symmetric.

\subsection{}
Let $v$ be an indeterminate. Set $v_i = v^{d_i}$ for each $i \in
I$. More generally, given any rational function $P \in \Q(v)$ we let
$P_i$ denote the rational function obtained from $P$ by replacing each
occurrence of $v$ by $v_i=v^{d_i}$.

Set $\A = \Z[v,v^{-1}]$. For $a\in \Z$, $t\in \N$ we set
\[
   \sqbinom{a}{t} = \prod_{s=1}^t \frac{v^{a-s+1} - v^{-a+s-1}} {v^s -
   v^{-s}}.
\]
A priori this is an element of $\Q(v)$, but actually it lies in $\A$
(see \cite[1.3.1(d)]{Lusztig}). We set
\[
  [n] = \sqbinom{n}{1} = \frac{v^n - v^{-n}}{v - v^{-1}}
\qquad (n \in \Z)
\]
and
\[
   [n]^! = [1]\cdots [n-1]\,[n] \qquad (n\in \N).
\]
Then it follows that
\[
\sqbinom{a}{t} = \frac{[a]^!}{[t]^!\,[a-t]^!}
\qquad \text{for all $0 \le t \le a$.}
\]

\subsection{} \label{def:UU}
Let $\UU$ be the quantized enveloping algebra associated to the given
root datum $(X,\{\alpha_i\}, Y,\{h_i\})$. Thus, $\UU$ is the
associative $\Q(v)$-algebra with 1 given by the generators
\[
  E_i \quad (i\in I), \qquad F_i \quad (i\in I), \qquad
  K_h \quad (h \in Y)
\]
and satisfying the relations

(a) $K_0 = 1$, $K_h K_{h'} = K_{h + h'}$ for all $h,h' \in Y$;

(b) $K_h E_i = v^{\bil{h}{\alpha_i}} E_i K_h$ for all $i \in I$,
$h \in Y$;

(c) $K_h F_i = v^{-\bil{h}{\alpha_i}} F_i K_h$ for all $i \in
I$, $h \in Y$;

(d) $E_i F_j - F_j E_i = \delta_{ij} \dfrac{\tilK_i -
\tilK_{-i}}{v_i-v_i^{-1}}$ for any $i,j \in I$;

(e) $\displaystyle \sum_{s+s'=1-\bil{h_i}{\alpha_j}} (-1)^{s'}
\divided{E_i}{s}E_j \divided{E_i}{s'} = 0$ for all $i \ne j$;

(f) $\displaystyle \sum_{s+s'=1-\bil{h_i}{\alpha_j}} (-1)^{s'}
\divided{F_i}{s}F_j \divided{F_i}{s'} = 0$ for all $i \ne j$.

\noindent
In (d) above we set $\tilK_i = K_{d_i h_i} = (K_{h_i})^{d_i}$,
$\tilK_{-i} = K_{-d_i h_i} = (K_{h_i})^{-d_i}$ and in (e), (f) above
we set $\divided{E_i}{s} = (E_i)^s/([s]_i^!)$, $\divided{F_i}{s} =
(F_i)^s/([s]_i^!)$.

The algebra $\UU$ admits a Hopf algebra structure, with coproduct
\[
\Delta: \UU \to \UU\otimes \UU
\]
given by the algebra homomorphism satisfying

$\Delta K_h = K_h \otimes K_h$;

$\Delta E_i = E_i\otimes \tilK_i^{-1} + 1 \otimes E_i$; \quad
$\Delta F_i = F_i\otimes 1 + \tilK_i \otimes E_i$

\noindent
for each $h\in Y$, $i \in I$. Using $\Delta$ one defines a
$\UU$-module structure on the tensor product of two given
$\UU$-modules in the usual manner.

\subsection{The classical case} \label{U:classical}
Setting $K_h = v^h$ and letting $v$ tend to 1 (which forces $K_h$ to
approach 1), the algebra $\UU$ becomes the universal enveloping
algebra $\U=\U(\g)$ (over $\Q$) of the corresponding Kac-Moody Lie
algebra $\g$ generated by $\mathfrak{h}=\Q \otimes_\Z Y$ and $\{e_i,
f_i: i \in I\}$, with the defining relations

(a) $[h,h']=0$;

(b) $[h, e_i] = \bil{h}{\alpha_i} e_i$;

(c) $[h, f_i] = -\bil{h}{\alpha_i} f_i$;

(d) $[e_i, f_j] = \delta_{ij} h_i$;

(e) $(\text{ad } e_i)^{1 - \bil{h_i}{\alpha_j}} e_j = 0 =  (
\text{ad } f_i)^{1 - \bil{h_i}{\alpha_j}} f_j \quad (i\ne j)$

\noindent
for all $h,h' \in \mathfrak{h}$ and all $i,j \in I$.

\subsection{Triangular decomposition} \label{td}
Let $\UU^0$ be the subalgebra of $\UU$ generated by the $K_h$, $h \in
Y$. It is clear that $\UU^0 \simeq \Q(v)[Y]$, the group algebra of
$Y$. Denote by $\UU^+$ (respectively, $\UU^-$) the subalgebra of $\UU$
generated by the $E_i$ (respectively, the $F_i$). Then the map
\[
  \UU^- \otimes \UU^0 \otimes \UU^+ \to \UU
\]
given by $x\otimes K_h \otimes y \to x K_h y$ is a vector space
isomorphism.

\subsection{}
There is a unique action of the Weyl group $W$ on $Y$ such that
$s_i(h)=h - \bil{h}{\alpha_i} h_i$ for all $i \in I$. Similarly, there
is a unique action of $W$ on $X$ such that $s_i(\lambda)=\lambda -
\bil{h_i}{\lambda} \alpha_i$ for all $i \in I$. Then
$\bil{s_i(h)}{\lambda} = \bil{h}{s_i(\lambda)}$ for all $h \in Y$,
$\lambda\in X$. Hence for any $w \in W$ we have
\[
  \bil{w(h)}{\lambda} = \bil{h}{w^{-1}(\lambda)}
\]
for all $h\in Y$, $\lambda \in X$.

\subsection{} \label{terminology}
A Cartan datum is of {\em finite type} if the symmetric matrix $(\,
(i,j)\, )$ indexed by $I\times I$ is positive definite. This is
equivalent to the requirement that $W$ is a finite group. A Cartan
datum that is not of finite type is said to be of {\em infinite type}.

A root datum is $X$-{\em regular}, respectively, $Y$-{\em regular} if
$\{\alpha_i\}$ (respectively, $\{h_i\}$) is linearly independent in $X$
(respectively, $Y$). We note for later reference that if the underlying
Cartan datum is of finite type then the root datum is automatically
both $X$-regular and $Y$-regular.

In the case where a root datum is $X$-regular, there is a partial
order on $X$ given by: $\lambda \le \lambda'$ if and only if $\lambda'
- \lambda \in \sum_i \N \alpha_i$.  In the case where the root datum
is $Y$-regular, we define $X^+ = \{ \lambda \in X \mid
\bil{h_i}{\lambda} \in \N, \text{ all } i \}$, elements of which are
known as dominant weights.

\subsection{The category $\mathcal{C}$} \label{modules}
Let $V$ be a $\UU$-module.  Given any $\lambda\in X$, set
\[
V_{\lambda} = \{ m\in V \mid K_h m =
v^{\bil{h}{\lambda}} m , \text{ for all } h \in Y \}.
\]
Let $\mathcal{C}$ be the category whose objects are $\UU$-modules $V$
such that $V = \oplus_{\lambda \in X} V_{\lambda}$ (as a $\Q(v)$
vector space).  Morphisms in $\mathcal{C}$ are $\UU$-module maps. The
subspaces $V_\lambda$ are called weight spaces of $V$; thus the
category $\mathcal{C}$ is the category of $\UU$-modules admitting a
type $\mathbf{1}$ weight space decomposition.

Let $\lambda \in X$. The Verma module $M(\lambda)$ is by definition
the quotient of $\UU$ by the left ideal $\sum_i \UU E_i + \sum_\mu
\UU(K_\mu - v^{\bil{\mu}{\lambda}})$. This is an object of
$\mathcal{C}$.

Recall that an object $M$ of $\mathcal{C}$ is {\em integrable} if for
any $m \in M$ and any $i \in I$ there is some positive integer $N$
such that $\divided{E_i}{n} m = \divided{F_i}{n} m = 0$ for all $n \ge
N$.

\subsection{The modules $L(\lambda)$}
Assume that the root datum is $Y$-regular. Then to each $\lambda \in
X^+$ corresponds an integrable object $L(\lambda)$ of $\mathcal{C}$,
which may be defined as the quotient of $M(\lambda)$ by the submodule
generated by all $\theta_i^{\bil{h_i}{\lambda}+1}$ for various $i \in
I$, where $\theta_i$ denotes the image of $E_i$ in $M(\lambda)$.  If
the root datum is both $X$-regular and $Y$-regular, then $L(\lambda)$
is a simple object in the full subcategory of $\mathcal{C}$ whose
objects are the integrable $\UU$-modules \cite[6.2.3]{Lusztig}.
Moreover, $L(\lambda)$ is not isomorphic to $L(\lambda')$ unless
$\lambda = \lambda'$ (for $\lambda, \lambda'$ in $X^+$). These
properties hold in case the root datum has finite type, since a root
datum of finite type is automatically both $X$-regular and
$Y$-regular.

Assume now that the root datum has finite type. Then $L(\lambda)$ is
of finite dimension over $\Q(v)$, and one has the following complete
reducibility property (see \cite[6.3.6]{Lusztig}): every integrable
$\UU$-module is a direct sum of simple $\UU$-modules isomorphic with
$L(\lambda)$ for various $\lambda \in X^+$. Since finite-dimensional
objects of $\mathcal{C}$ are integrable, they are completely reducible
in the above sense.

\section{Idempotent presentation of $\Sch(\pi)$}\label{app:A}\noindent
In this section we generalize the presentation of \cite{PGSA} to any
generalized $q$-Schur algebra corresponding to a given root datum
$(X,\{\alpha_i\}, Y, \{h_i\})$ of finite type and a given finite
saturated set $\pi \subset X^+$. In particular, this extends the
results of \cite{PGSA} to the reductive case.

\subsection{Generalized $q$-Schur algebras} \label{gsa}
Fix $(X,\{\alpha_i\}, Y,\{h_i\})$, a root datum whose underlying
Cartan datum is of finite type.  Recall (see \ref{terminology}) that
in this case the root datum is both $X$-regular and $Y$-regular.
A given subset $\pi$ of $X^+$ is said to be {\em saturated} (with
respect to the partial order $\le$ defined in \ref{terminology}) if
$\lambda \in \pi$ and $\lambda' \in X^+$ with $\lambda' \le \lambda$
imply $\lambda' \in \pi$.

Given a saturated subset $\pi$ of $X^+$ we define $\mathcal{C}[\pi]$
to be the full subcategory of $\mathcal{C}$ (see \ref{modules}) whose
objects are the $\UU$-modules $M$ in $\mathcal{C}$ such that $M =
\bigoplus_{\lambda \in W\pi} M_\lambda$.  Every finite dimensional
object of $\mathcal{C}[\pi]$ satisfies the property: every simple
composition factor of $M$ is isomorphic to some $L(\lambda)$ for
$\lambda \in \pi$.

A {\em generalized $q$-Schur algebra} is a quotient of the form
$\Sch(\pi) = \UU/\ideal$, where $\ideal$ is the ideal of $\UU$
consisting of the elements of $\UU$ annihilating every object of
$\mathcal{C}[\pi]$.  The ideal $\ideal$ is the {\em defining ideal} of
the generalized $q$-Schur algebra $\Sch(\pi)$.

\subsection{The algebra $\SchP(\pi)$}\label{def:Spi}
Given a root datum $(X,\{\alpha_i\}, Y,\{h_i\})$ and a finite
saturated set $\pi \subset X^+$ we define an algebra $\SchP(\pi)$ to be
the associative algebra with 1 given by the generators
\[
 E_i \quad(i \in I), \qquad F_i \quad(i \in I),\qquad  1_\lambda
 \quad(\lambda \in W\pi)
\]
and the relations

(a) $1_\lambda 1_{\lambda'} = \delta_{\lambda,\lambda'} 1_\lambda$, \qquad
$\sum_{\lambda \in W\pi} 1_\lambda = 1$;

(b) $E_i 1_\lambda = 1_{\lambda+\alpha_i} E_i$,\qquad $1_\lambda E_i =
E_i 1_{\lambda-\alpha_i}$;

(c) $F_i 1_\lambda = 1_{\lambda-\alpha_i} F_i$,\qquad $1_\lambda F_i =
F_i 1_{\lambda+\alpha_i}$;

(d) $E_iF_j - F_j E_i = \delta_{ij} \sum_{\lambda \in W\pi}
\;[\bil{h_i}{\lambda}]_i \, 1_\lambda$;

(e) $\displaystyle \sum_{s+s'=1-\bil{h_i}{\alpha_j}} (-1)^{s'}
\divided{E_i}{s}E_j \divided{E_i}{s'} = 0$ for $i \ne j$;

(f) $\displaystyle \sum_{s+s'=1-\bil{h_i}{\alpha_j}} (-1)^{s'}
\divided{F_i}{s}F_j \divided{F_i}{s'} = 0$ for $i \ne j$

\noindent
for all $i,j \in I$ and all $\lambda, \lambda' \in W\pi$.

For (b), (c) above, one must interpret the symbol $1_{\lp}$ as zero
whenever $\lp \notin W\pi$.

\subsection{}
We claim that $\SchP(\pi)$ is isomorphic with the generalized
$q$-Schur algebra $\Sch(\pi) = \UU/\ideal$ corresponding to $\pi$.
This will eventually be proved in \ref{thm:idemp-present} ahead, after
a series of lemmas to prepare the way. From now on fix $\pi$ and write
$\SchP = \SchP(\pi)$.

For any $h \in Y$ define an element $K_h \in \SchP$ by
\[
\textstyle K_h = \sum_{\lambda \in W\pi} v^{\bil{h}{\lambda}} 1_\lambda.
\]
Then in particular $K_0 = 1$.  An easy calculation shows that $K_h
K_{h'} = K_{h+h'}$ for all $h, h' \in Y$. Thus the elements $K_h$ ($h
\in Y$) satisfy \ref{def:UU}(a).

\begin{lem}\label{lem:3part}
Let $\{ H_1, \dots, H_n \}$ be a basis of the free
abelian group $Y$, and let $\SchP^0$ be the subalgebra of $\SchP$
generated by $K_{H_1}, \dots, K_{H_n}$. (This is
the same as the subalgebra generated by all $K_h$ for $h \in Y$.)

(a) $1_\lambda \in \SchP^0$ for all $\lambda \in W\pi$.

(b) $K_{-H_1}, \dots, K_{-H_n} \in \SchP^0$.

(c) $\{ 1_\lambda \mid \lambda \in W\pi \}$ is a basis for $\SchP^0$.
\end{lem}

\begin{proof}
Let $\{ \varepsilon^\prime_1, \dots, \varepsilon^\prime_n \}$ be the
basis of $X$ dual to $\{ H_1, \dots, H_n \}$, so
that $\bil{H_a}{\varepsilon^\prime_b} = \delta_{ab}$ for
$a,b \in \{1, \dots, n\}$. For $\lambda \in X$ write
\[\textstyle
\lambda = \sum_{a=1}^n \lambda_a \,\varepsilon^\prime_a \qquad
(\lambda_a \in \Z).
\]
We have $\lambda_a = \bil{H_a}{\lambda}$ for each $a$.

Let $\lambda \in X$ be given.  For each $a=1, \dots, n$ set
$\Gamma(a,\lambda) = \{ \lp \in W\pi \mid \lp_a = \lambda_a \}$.  Set
\[
J_a^\lambda = \prod_{\stackrel{\lpp\in W\pi}{\lpp \notin \Gamma(a,\lambda)}}
(K_{H_a} - v^{\lpp_a})
\]
We have equalities
\begin{equation*}
\begin{aligned}
J_a^\lambda &= \prod_{\lpp} \left( \sum_{\lp \in W\pi} v^{\lp_a}1_{\lp} -
v^{\lpp_a} \sum_{\lp \in W\pi} 1_{\lp} \right) \\
&=  \prod_{\lpp} \left( \sum_{\lp \in W\pi} (v^{\lp_a} -
v^{\lpp_a}) 1_{\lp} \right) \\
&= \sum_{\lp \in W\pi} \prod_{\lpp}  (v^{\lp_a} -
v^{\lpp_a}) 1_{\lp}
\end{aligned}
\end{equation*}
where the products are over all $\lpp \in W\pi -
\Gamma(a,\lambda)$. The idempotent orthogonality relations \ref{def:Spi}(a)
were used to interchange the sum and product.  Noting that the product
in the sum on the last line above vanishes for any $\lp \in W\pi
- \Gamma(a,\lambda)$, we obtain the expression
\begin{equation*}
J_a^\lambda = \sum_{\lp \in \Gamma(a,\lambda)} \prod_{\lpp}
(v^{\lp_a} - v^{\lpp_a}) 1_{\lp}
\end{equation*}
where the product in this sum is a non-zero {\em constant}, since
$\lp_a = \lambda_a$ for all $\lp \in \Gamma(a,\lambda)$.  This proves
that $J_a^\lambda$ is (up to a non-zero scalar) the sum of all
idempotents $1_{\lp}$ for which $\lp_a = \lambda_a$. This property
holds for all $a=1, \dots, n$. It follows that the product
$J_1^\lambda \cdots J_n^\lambda$ is, up to a non-zero scalar multiple,
equal to $1_\lambda$, since $1_\lambda$ is the unique idempotent
appearing in each of the sums in the product.

By definition $J_a^\lambda$ belongs to $\SchP^0$, so the result of the
previous paragraph shows that $1_\lambda$ (for any $\lambda \in W\pi$)
lies within $\SchP^0$.  This proves part (a).

Part (b) follows from part (a) and the definition of
$K_{-H_a}$.

By definition of the $K_{H_a}$ we see that the subalgebra of
$\SchP$ generated by the $1_\lambda$ ($\lambda \in W\pi$) contains the
$K_{H_a}$. By part (a) this subalgebra equals $\SchP^0$. Part
(c) now follows from the fact that the $1_\lambda$ ($\lambda \in
W\pi$) form a family of orthogonal idempotents.
\end{proof}

\begin{lem}\label{image}
The algebra $\SchP=\SchP(\pi)$ is a homomorphic image of $\UU$, via the
homomorphism sending $E_i \to E_i$, $F_i \to F_i$, $K_h \to K_h$.
\end{lem}

\begin{proof}
The algebra $\SchP(\pi)$ is by definition generated by all $E_i, F_i$
($i\in I$) and $K_h$ ($h \in Y$).  We have already observed that
the generators of $\SchP$ satisfy \ref{def:UU}(a).

We remind the reader of the convention $1_\lambda = 0$ for any
$\lambda \in X - W\pi$. For any $h \in Y$ we have
\begin{equation*}
\begin{aligned}
K_h E_i &= \sum_{\lambda\in X} v^{\bil{h}{\lambda}}
1_\lambda E_i = \sum_{\lambda\in X}
v^{\bil{h}{\lambda}} E_i 1_{\lambda-\alpha_i}\\
&= \sum_{\lambda\in X}
v^{\bil{h}{\lambda+\alpha_i}} E_i 1_{\lambda}
= v^{\bil{h}{\alpha_i}} E_i \sum_{\lambda\in X}
v^{\bil{h}{\lambda}} 1_{\lambda}
= v^{\bil{h}{\alpha_i}} E_i K_h
\end{aligned}
\end{equation*}
which proves that the generators of $\SchP$ satisfy \ref{def:UU}(b). An
entirely similar calculation proves that the generators of $\SchP$
satisfy \ref{def:UU}(c).

From \ref{def:Spi}(d) and the definitions we obtain equalities
\begin{equation*}
\begin{aligned}
E_i F_j - F_j E_i &= \delta_{ij} \sum_{\lambda \in W\pi}
[\bil{h_i}{\lambda}]_i 1_\lambda
= \delta_{ij} \sum_{\lambda \in W\pi}
\frac{v_i^{\bil{h_i}{\lambda}} - v_i^{-\bil{h_i}{\lambda}}}
{v_i - v_i^{-1}} \,1_\lambda \\
&= \delta_{ij} \frac{(\sum_{\lambda \in W\pi} v_i^{\bil{h_i}{\lambda}}
1_\lambda) - (\sum_{\lambda \in W\pi}
v_i^{-\bil{h_i}{\lambda}} 1_\lambda)}{v_i-v_i^{-1}} .
\end{aligned}
\end{equation*}
But $\sum_{\lambda \in W\pi} v_i^{\bil{h_i}{\lambda}} 1_\lambda =
\sum_{\lambda \in W\pi} v^{d_i\bil{h_i}{\lambda}} 1_\lambda =
(\sum_{\lambda \in W\pi} v^{\bil{h_i}{\lambda}} 1_\lambda)^{d_i} =
K_i^{d_i} = \tilK_i$ and by a similar calculation $\sum_{\lambda \in
W\pi} v_i^{-\bil{h_i}{\lambda}} 1_\lambda = \tilK_{-i}$.  Thus
the above equalities take the form
\[
E_i F_j - F_j E_i = \delta_{ij} \frac{ \tilK_i - \tilK_{-i} }{v_i-v_i^{-1}}
\]
proving that the generators of $\SchP$ satisfy \ref{def:UU}(d).

Relations \ref{def:UU}(e), (f) are also satisfied by the generators of
$\SchP$ since those relations are identical with \ref{def:Spi}(e),
(f). The lemma is proved.
\end{proof}

\begin{lem} \label{lem:fdss}
The algebra $\SchP=\SchP(\pi)$ is a finite-dimensional semisimple algebra.
\end{lem}

\begin{proof}
First, we note that the generators $E_i, F_i$ for all $i \in I$ are
nilpotent elements of $\SchP$, since by \ref{def:Spi}(c), (d) for $N$
sufficiently large we have $E_i^N 1_\lambda = 0 = F_i^N 1_\lambda$ for
all $\lambda \in W\pi$. This implies $E_i^N = 0 = F_i^N$ since $\sum
1_\lambda = 1$.

By the triangular decomposition (see \ref{td}) we have $\UU = \UU^-
\UU^0 \UU^+$.  It follows that $\SchP$ has a similar decomposition
$\SchP = \SchP^- \SchP^0 \SchP^+$ where $\SchP^0$, $\SchP^-$, $\SchP^+$ are
defined as the homomorphic images of $\UU^0$, $\UU^-$, $\UU^+$,
respectively, under the canonical quotient map $\UU \to \SchP$.

There is an analogue of the Poincare-Birkhoff-Witt (PBW) theorem for
$\UU$.  There exist (in addition to the given $E_i$ and $F_i$) root
vectors corresponding to each non simple root in the root system and
these additional root vectors can easily be shown to satisfy (in
$\SchP$) a commutation relation analogous to \ref{def:Spi}(c), (d).
Thus it follows that all the root vectors are nilpotent elements of
$\SchP$.  From the PBW theorem it follows that $\UU^+$ has a basis
consisting of products of powers of root vectors, taken in some fixed
ordering of the positive roots.  The nilpotence of the root vectors
thus implies that $\SchP^+$ is finite-dimensional. A similar argument
shows that $\SchP^-$ is finite-dimensional.

Another approach to the finite-dimensionality of $\SchP^+$, $\SchP^-$
can be obtained from the so-called `monomial' basis of $\UU^+$ (see
Lusztig \cite{Lusztig:1990} for the simply-laced case and Chari and Xi
\cite{CX} in general).

It follows from \ref{lem:3part}(c) that $\SchP^0$ is
finite-dimensional.  Thus, from the triangular decomposition for
$\SchP$ it follows immediately that $\SchP$ is finite-dimensional.

The algebra $\SchP$ is a $\UU$-module via the canonical quotient map
$\UU \to \SchP$. It is known that finite-dimensional $\UU$-modules are
completely reducible (i.e., a direct sum of simple
$\UU$-modules). Hence, $\SchP$ is semisimple as a $\UU$-module, hence
semisimple as an $\SchP$-module. Thus $\SchP$ is a semisimple algebra.
\end{proof}

\begin{remark}
  Although $\SchP$ is a direct sum of simple $\UU$-modules by the
  preceding argument, we do not yet know that it is an object of
  $\mathcal{C}$. That is a consequence of the next result.
\end{remark}

\begin{lem}\label{wtspace}
If $M$ is any finite-dimensional $\SchP$-module then $M$ is an object
of $\mathcal{C}[\pi]$ and the decomposition $M = \oplus_{\lambda\in
W\pi} 1_\lambda M$ is a weight space decomposition of $M$ as a
$\UU$-module.

\end{lem}

\begin{proof}
View $M$ as a $\UU$-module by means of the canonical quotient map
$\UU\to \SchP$.  In $\SchP$ we have the equality $1=\sum_{\lambda\in
W\pi} 1_\lambda$, which implies that $M = \oplus_{\lambda\in W\pi}
1_\lambda M$.  Since
\begin{equation*}
 1_\lambda K_h = K_h 1_\lambda = v^{\bil{h}{\lambda}} 1_\lambda
 \qquad (h \in Y)
\end{equation*}
we see immediately that $K_h \, m = v^{\bil{h}{\lambda}} m$ for any
$m \in 1_\lambda M$. This justifies the inclusion $1_\lambda M \subset
M_\lambda$, where
\[
M_\lambda = \{ m \in M \mid K_h \, m = v^{\bil{h}{\lambda}} m,
\text{ all } h \in Y \} .
\]

On the other hand, for a given $\lambda \in W\pi$, assume that $m \in
M_{\lambda}$, so $K_h m = v^{\bil{h}{\lambda}} m$ for all $h \in
Y$. Multiplying by $1_{\lp}$ we obtain
\[
1_{\lp} K_h  m = v^{\bil{h}{\lambda}} 1_{\lp} m
\]
and it follows that
\begin{equation*}
v^{\bil{h}{\lp}} 1_{\lp} m = v^{\bil{h}{\lambda}} 1_{\lp} m
\qquad(h \in Y, \lp \in W\pi).
\end{equation*}
Choosing $h \in Y$ judiciously one concludes that $1_{\lp} m = 0$
for any $\lp \ne \lambda$, and hence
\[
\textstyle
m = 1\cdot m = \sum_{\lp \in W\pi} 1_{\lp} m = 1_\lambda m.
\]
This proves that $m \in 1_\lambda M$, establishing the reverse
inclusion $M_{\lambda} \subset 1_\lambda M$.  Hence $M_{\lambda} =
1_\lambda M$ for all $\lambda \in W\pi$, and thus $M =
\oplus_{\lambda\in W\pi} 1_\lambda M = \oplus_{\lambda\in W\pi}
M_\lambda$, so $M$ is an object of $\mathcal{C}[\pi]$, as desired.
\end{proof}

\begin{lem}\label{iso}
The set $\{ L(\lambda) \mid \lambda \in \pi \}$ is the set of
isomorphism classes of simple $\SchP$-modules, and $\dim \SchP =
\sum_{\lambda \in \pi} (\dim L(\lambda))^2$.
\end{lem}

\begin{proof}
The simple $\SchP$-modules are necessarily simple objects of
$\mathcal{C}$. Let $\lambda \in X^+ - \pi$. If $L(\lambda)$
was an $\SchP$-module, then by \ref{wtspace} $L(\lambda)$ would be a
direct sum of the weight spaces $L(\lambda)_{\lp}$ as $\lp$ varies
over $W\pi$. This is a contradiction since $L(\lambda)_\lambda \ne 0$
and $\lambda \notin W\pi$.

On the other hand, for every $\lambda \in \pi$, $L(\lambda)$ inherits
a well-defined $\SchP$-module structure from its $\UU$-module
structure, just by defining the action on the generators of $\SchP$ by
the obvious formulas. The first claim is proved.

The second claim follows immediately by standard theory of
finite-dimensional algebras.
\end{proof}

\begin{thm}\label{thm:idemp-present}
  The algebra $\SchP(\pi)$ is isomorphic with the the generalized Schur
  algebra determined by the given root datum and the saturated set
  $\pi$. In other words, $\SchP(\pi) \simeq \UU/\ideal$ where $\ideal$ is the
  ideal consisting of all elements of $\UU$ annihilating every object
  of the category $\mathcal{C}[\pi]$.
\end{thm}

\begin{proof}
Let $A$ be the kernel of the canonical
quotient map $\UU \to \SchP$. From \ref{iso} it is clear that $A$
annihilates every object of $\mathcal{C}[\pi]$. Hence $A \subset \ideal$
where $\ideal$ is the ideal consisting of all $u \in \UU$ annihilating
every object of $\mathcal{C}[\pi]$, and we have a natural quotient map
$\SchP \simeq \UU/A \to \UU/\ideal$.

It follows that $\UU/\ideal$ is a finite-dimensional semisimple
algebra.  Each $L(\lambda)$ for $\lambda \in \pi$ is a simple
$\UU/\ideal$-module, so the dimension of $\UU/\ideal$ is at least
$\sum_{\lambda \in \pi} (\dim L(\lambda))^2$. But the existence of the
above quotient map $\SchP \to \UU/\ideal$ ensures that the dimension
of $\UU/\ideal$ is at most $\sum_{\lambda \in \pi} (\dim
L(\lambda))^2$.  This proves equality of dimension between $\SchP$ and
$\UU/\ideal$, so the quotient map $\SchP \to \UU/\ideal$ must be an
isomorphism, and the result is proved.
\end{proof}

The theorem shows that $\Sch(\pi) \simeq \SchP(\pi)$; hence the
generators and relations of \ref{def:Spi} provide a presentation of
$\Sch(\pi)$.

\begin{cor}
  Let $\dot{\UU}$ be Lusztig's modified form of $\UU$. Then
$\Sch(\pi)$ is isomorphic with $\dot{\UU}/\dot{\UU}[X^+ - \pi]$
(notation of \cite[29.2]{Lusztig}); hence $\Sch(\pi)$ inherits a
canonical basis from the canonical basis of $\dot{\UU}$.
\end{cor}

\begin{proof}
  The proof is entirely similar to the proof of \cite[4.2, 4.3]{PGSA}.
Details are left to the reader.
\end{proof}

\section{Idempotent presentation of $S(\pi)$}\label{app:B}\noindent
This section treats the classical analogue of the presentation given
in the previous section. The defining relations in this case are
obtained simply by setting $v=1$ in the defining relations of
\ref{def:Spi}.

This generalizes the presentation of \cite{PGSA} in the $v=1$ case
(the classical case) to any generalized Schur algebra corresponding to
a given root datum $(X,\{\alpha_i\}, Y, \{h_i\})$ of finite type and a
given finite saturated set $\pi \subset X^+$. In particular, this
applies to the reductive case.

\subsection{Generalized Schur algebras} \label{classical:gsa}
Fix a root datum $(X,\{\alpha_i\}, Y,\{h_i\})$ of finite type.  Let
$\U = \U(\g)$ denote the classical universal enveloping algebra
corresponding to the given root datum (see \ref{U:classical}).  Let
$\pi$ be a saturated subset of $X^+$.

Let $\mathcal{C}$ to be the category of $\U$-modules admitting a
weight space decomposition, and let $\mathcal{C}[\pi]$ be the full
subcategory of $\mathcal{C}$ whose objects are the $\U$-modules $M$
for which $M = \bigoplus_{\lambda \in W\pi} M_\lambda$.  Every finite
dimensional object of $\mathcal{C}[\pi]$ satisfies the property: every
simple composition factor of $M$ is isomorphic to some $L(\lambda)$
for $\lambda \in \pi$.

A {\em generalized Schur algebra} \cite{SA1} is a quotient of the form
$S(\pi) = \U/\ideal$, where $\ideal$ is the ideal of $\U$ consisting
of the elements of $\U$ annihilating every object of
$\mathcal{C}[\pi]$.  The ideal $\ideal$ is the {\em defining ideal} of
the generalized Schur algebra $S(\pi)$.

\subsection{The algebra $\widetilde{S}(\pi)$}\label{def:Spiv=1}
Given a root datum $(X,\{\alpha_i\}, Y,\{h_i\})$ and a finite
saturated set $\pi \subset X^+$ let $\widetilde{S}(\pi)$ be the
associative algebra over $\Q$ with 1 given by the generators
\[
 e_i \quad(i \in I), \qquad f_i \quad(i \in I),\qquad  1_\lambda
 \quad(\lambda \in W\pi)
\]
and the relations

(a) $1_\lambda 1_{\lambda'} = \delta_{\lambda,\lambda'} 1_\lambda$, \qquad
$\sum_{\lambda \in W\pi} 1_\lambda = 1$;

(b) $e_i 1_\lambda = 1_{\lambda+\alpha_i} e_i$,\qquad $1_\lambda e_i =
e_i 1_{\lambda-\alpha_i}$;

(c) $f_i 1_\lambda = 1_{\lambda-\alpha_i} f_i$,\qquad $1_\lambda f_i =
f_i 1_{\lambda+\alpha_i}$;

(d) $e_if_j - f_j e_i = \delta_{ij} \sum_{\lambda \in W\pi}
 \;\bil{h_i}{\lambda} \, 1_\lambda$;

(e) $\displaystyle \sum_{s+s'=1-\bil{h_i}{\alpha_j}} (-1)^{s'}
\divided{e_i}{s}e_j \divided{e_i}{s'} = 0$ for $i \ne j$;

(f) $\displaystyle \sum_{s+s'=1-\bil{h_i}{\alpha_j}} (-1)^{s'}
\divided{f_i}{s}f_j \divided{f_i}{s'} = 0$ for $i \ne j$

\noindent
for all $i,j \in I$ and all $\lambda, \lambda' \in W\pi$.

For (b), (c) above, one must interpret the symbol $1_{\lp}$ as zero
whenever $\lp \notin W\pi$.  In (e), (f) one interprets
$\divided{e_i}{s} = e_i^{s}/(s!)$, $\divided{f_i}{s} = f_i^{s}/(s!)$
as ordinary divided powers.

\begin{thm} There is an algebra isomorphism $S(\pi) \to
  \widetilde{S}(\pi)$.  In other words, $\widetilde{S}(\pi) \simeq
  S(\pi) = \U/\Lambda$ where $\Lambda$ is the ideal of $\U$ consisting
  of all elements $u$ annihilating all objects of $\mathcal{C}[\pi]$.
\end{thm}

One can prove this by following a parallel approach to the line of
argument given in the previous section, the arguments being similar,
but easier, in this instance. We omit the details.

\section{Main result}\noindent
We now formulate and prove the main result of this paper, see
\ref{thm:main}, which reduces the problem of finding relations to
define $\Sch(\pi)$ as a quotient of $\UU$ to the problem of finding
equations that cut out a set of points corresponding to the finite set
$W\pi$ of weights, as an affine variety.

\subsection{} \label{def:UUalt}
Recall (see \ref{td}) that $\UU^0$ is isomorphic with the group
algebra $\Q(v)[Y]$ of the free abelian group $Y$. Fix an arbitrary
basis $\{H_1, \dots, H_n \}$ of $Y$. From
\ref{def:UU}(a) it follows that the elements
\[
K_{H_1}, \dots, K_{H_n}, K_{-H_1},
\dots, K_{-H_n}
\]
generate $\UU^0$ as an algebra over $\Q(v)$. Moreover, it is easy to
check that relations \ref{def:UU}(b), (c) for general $h \in Y$ are
consequences of those same relations imposed on just the $h =
H_1, \dots, H_n$. Hence, $\UU$ is the
associative $\Q(v)$-algebra with 1 given by the generators
\[
  E_i \quad (i\in I), \qquad F_i \quad (i\in I), \qquad
  K_{H_1}, \dots, K_{H_n}, K_{-H_1},
\dots, K_{-H_n}
\]
and satisfying the relations

(a) $K_{H_a} K_{H_b} = K_{H_b}K_{H_a}$, $K_{H_a} K_{-H_a} = 1$ for
$a,b=1, \dots, n$;

(b) $K_{H_a} E_i = v^{\bil{H_a}{\alpha_i}} E_i
K_{H_a}$ for $i \in I$, $a=1, \dots, n$;

(c) $K_{H_a} F_i = v^{-\bil{H_a}{\alpha_i}} F_i
K_{H_a}$ for $i \in I$, $a=1, \dots, n$;

(d) $E_i F_j - F_j E_i = \delta_{ij} \dfrac{\tilK_i -
\tilK_{-i}}{v_i-v_i^{-1}}$ for any $i,j \in I$;

(e) $\displaystyle \sum_{s+s'=1-\bil{h_i}{\alpha_j}} (-1)^{s'}
\divided{E_i}{s}E_j \divided{E_i}{s'} = 0$ for all $i \ne j$;

(f) $\displaystyle \sum_{s+s'=1-\bil{h_i}{\alpha_j}} (-1)^{s'}
\divided{F_i}{s}F_j \divided{F_i}{s'} = 0$ for all $i \ne j$.

\subsection{}
We have an isomorphism of $\UU^0$ with the ring of Laurent polynomials
\[
\UU^0 \simeq \Q(v)[K_{H_1}, \dots, K_{H_n},
K_{-H_1}, \dots, K_{-H_n}].
\]
The polynomial algebra $\Q(v)[K_{H_1}, \dots,
K_{H_n}]$ is a subalgebra of $\UU^0$. We find it convenient
to regard $\Q(v)[K_{H_1}, \dots, K_{H_n}]$ as the
affine algebra of the affine variety $\Aff_{\Q(v)}^n$, and we will henceforth
regard $K_{H_a}$ for $a=1, \dots, n$ as coordinate functions
on $\Aff_{\Q(v)}^n$.

Any subset $D \subset X$ determines a corresponding set of points
\[
 P_D = \{ (v^{\bil{H_1}{\lambda}}, \dots,
  v^{\bil{H_n}{\lambda}}) \mid \lambda \in D \}
\]
in $\Aff_{\Q(v)}^n$.  If $D$ is finite then so is $P_D$. In that case, $P_D$
may be regarded as an affine variety in $\Aff_{\Q(v)}^n$.  Assuming that $D$
is finite, let $I(P_D)$ be the vanishing ideal of the set $P_D$;
$I(P_D)$ is an ideal of $\Q(v)[K_{H_1}, \dots, K_{H_n}] \subset
\UU^0$. In particular, we have the ideal $I(P_{W\pi})$.

\begin{lem} \label{lem:one}
The ideal $I(P_{W\pi})$ is contained in the defining ideal $\ideal =
\{u\in \UU \mid u \text{ acts as zero on all objects of
$\mathcal{C}[\pi]$} \}$.
\end{lem}

\begin{proof}
Let $f(K_{H_1}, \dots, K_{H_n})$ be an element of
$\Q(v)[K_{H_1}, \dots, K_{H_n}]$ which vanishes on
every point of the affine variety $P_{W\pi} \subset \Aff_{\Q(v)}^n$. Every
$K_h$ ($h \in Y$) acts on vectors in the $\lambda$ weight space
$M_\lambda$ of any object $M$ of $\mathcal{C}$, for any $\lambda \in
X$, as the scalar $v^{\bil{h}{\lambda}}$; thus the function
$f(K_{H_1}, \dots, K_{H_n})$ acts on such a weight
space as the scalar $f(v^{\bil{H_1}{\lambda}}, \dots,
v^{\bil{H_n}{\lambda}})$, which is zero for any $\lambda \in
W\pi$. Thus any element of $I(P_{W\pi})$ acts as zero on all the
weight spaces of any object of $\mathcal{C}[\pi]$, whence the result.
\end{proof}

\subsection{}\label{qm}
Let $J = \sum_{a\in I(P_{W\pi})} \UU a \UU$ be the ideal of $\UU$
generated by $I(P_{W\pi})$. By the preceding lemma, $J \subset
\ideal$, so there is a surjective quotient map $\UU/J \to \UU/\ideal =
\Sch(\pi)$.  For ease of notation, set $\T = \UU/J$.  Our task is to
show that the quotient map $\T \to \Sch(\pi)$ is an isomorphism.  This
will be accomplished eventually in \ref{thm:main} below.  The strategy
of proof is to produce a surjection $\theta: \Sch(\pi)\to \T$ by
showing that $\T$ has a set of generators satisfying the defining
relations of $\Sch(\pi)$, given in \ref{def:Spi}.

We begin with an examination of a certain quotient of
\[
\UU^0=\Q(v)[K_{\pm H_1},\ldots, K_{\pm H_n}]
\]
determined by an arbitrary finite set $P$ of points in $\Aff_{\Q(v)}^n$.

\begin{lem}\label{lem:two}
Let $P$ be any finite set of points in $\Aff_{\Q(v)}^n$ such that
$K_{H_a}(p)\neq 0$ for all $a=1,\ldots, n$ and all $p\in
P$. Let $\phi: \UU^0\rightarrow \bigoplus_{p\in P}\Q(v)$ be the
evaluation homomorphism sending $f\in \UU^0$ to the vector
$(f(p))_{p\in P}$.  Then $\ker \phi = I'(P)$ where $I'(P)$ is the
ideal of $\UU^0$ generated by $I(P)$.
\end{lem}

\begin{proof}
Evidently $I'(P)\subset \ker \phi$, so it is enough to establish the
reverse inclusion. Let $f \in \Q(v)[K_{\pm H_1},\ldots, K_{\pm H_n}]$
be a Laurent polynomial such that $f(K_{H_1}(p),\ldots, K_{H_n}(p))=0$
for all $p\in P$. There exist negative integers $m_1,\ldots,m_n$ such
that $f = (K_{H_1}^{m_1}\cdots K_{H_n}^{m_n})g$ with $g\in
\Q(v)[K_{H_1},\ldots, K_{H_n}]$. Then for all $p \in P$ we have
\[
\begin{aligned}
0 &= f(K_{H_1}(p),\ldots, K_{H_n}(p)) \\
  &=(K_{H_1}(p)^{m_1}\cdots K_{H_n}(p)^{m_n})g(K_{H_1}(p),\ldots
K_{H_n}(p)).
\end{aligned}
\]
Since $K_i(p)\neq 0$ for $i=1,\ldots n$ it follows that
$g(K_{H_1}(p),\ldots K_{H_n}(p))=0$ for all $p\in P$. Thus $g\in
I(P)$ and hence $f\in I'(P)$.
\end{proof}

\subsection{} \label{rmks}
Thus we have an explicit isomorphism $\UU^0/I'(P)\simeq
\bigoplus_{p\in P}\Q(v)$, for any $P \subset \Aff_{\Q(v)}^n$ satisfying the
condition of the lemma. Let $1_p$ be an element of
$\Q(v)[K_{H_1},\ldots,K_{H_n}]\subset \UU^0$ satisfying the
condition $1_p(q)=\delta_{pq}$ for all $q\in P$. Then the elements
$\{1_p\mid p\in P\}$ are mutually orthogonal idempotents and
$\sum_p 1_p$ is the identity element of $\UU^0/I'(P)$.

The algebra $\T$ admits a triangular decomposition, $\T=\T^-\T^0\T^+$,
where $\T^-$, $\T^0$, $\T^+$ are defined respectively as the image of
$\UU^-$, $\UU^0$, $\UU^+$. The canonical quotient map $\UU \rightarrow
\UU/J$ induces a map $\UU^0\rightarrow \UU/J$ whose kernel is
$\UU^0\cap J$.  Thus $\T^0\simeq \UU^0/(\UU^0\cap J)$. Clearly
$I'(P_{W\pi})\subset \UU^0\cap J$.  Thus we have a sequence of algebra
surjections:
\[
\UU^0/I'(P_{W\pi}) \rightarrow \T^0 \rightarrow \Sch^0(\pi).
\]
(The definition of $\Sch^0(\pi)$ appeared in the proof of
\ref{lem:fdss}.)  The vector space dimension of $\UU^0/I'(P_{W\pi})$ and
$\Sch^0(\pi)$ are both equal to the cardinality of $W\pi$; hence the
above surjections are algebra isomorphisms.

\subsection{}\label{PGSArelations}

Now let $E_i,F_i,K_h$ ($i \in I, h\in Y$) denote the images of the
respective elements of $\UU$ under the quotient map $\UU \to
\T=\UU/J$. Denote by $1_{\lambda}$ the idempotent $1_{p_\lambda}$
corresponding to the point
\[
p_\lambda = (v^{\bil{H_1}{\lambda}}, \dots,
v^{\bil{H_n}{\lambda}}).
\]
The $\{ 1_\lambda \mid \lambda \in W\pi \}$ form a basis for the
algebra $\T^0$. Recall that we are seeking a surjection $\theta:
\Sch(\pi)\to \T$. To produce such a surjection, it is enough to show
that the elements $E_i$, $F_i$, $1_{\lambda}$ ($i \in I$, $\lambda \in
W\pi$) satisfy the defining relations \ref{def:Spi}(a)--(f).

\begin{lem}
For all $h \in Y$ the identity $K_h = \sum_{\lambda \in W\pi}
v^{\bil{h}{\lambda}} 1_\lambda$ holds in the algebra $\T^0$.
\end{lem}

\begin{proof}
According to \ref{rmks} in light of the identification $\T^0\simeq
\UU^0/(I'(P_{W\pi}))$, we see that
\begin{equation}\label{eq:1} \textstyle
K_{H_a} = \sum_{\lambda\in W\pi} K_{H_a}(p_\lambda) 1_\lambda
= \sum_{\lambda\in W\pi} v^{\bil{H_a}{\lambda}} 1_\lambda,
\end{equation}
an identity in the algebra $\T^0$ for $a=1,\dots, n$.  From this it
follows that
\begin{equation}\label{eq:2} \textstyle
K_{-H_a} = \sum_{\lambda\in W\pi} v^{-\bil{H_a}{\lambda}}
1_\lambda.
\end{equation}
Indeed, one easily verifies that the product of the right hand side of
\eqref{eq:1} with the the right hand side of \eqref{eq:2} is equal to
1 in $\T^0$.  Write $h = \sum_{a=1}^n z_a H_a$ where $z_a \in
\Z$. Then from \ref{def:UU}(a) it follows that
\begin{equation} \label{eq:3}
K_h = \prod_{\{a \mid z_a > 0\}} (K_{H_a})^{z_a} \;
\prod_{\{a \mid z_a < 0\}} (K_{-H_a})^{-z_a}
\end{equation}
where all exponents on the right hand side are positive integers. But for
any positive integer $t$ one has
\[
K_{\pm H_a}^t = \bigg( \sum_{\lambda \in W\pi} v^{\bil{\pm H_a}{\lambda}}
1_\lambda \bigg)^t = \sum_{\lambda \in W\pi} v^{\bil{\pm tH_a}{\lambda}}
1_\lambda
\]
using the orthogonality of the the $1_\lambda$. The result follows
from this and \eqref{eq:3}, once again using the orthogonality of the
$1_\lambda$.
\end{proof}

\begin{lem} \label{lem:d}
  The identity $E_iF_j-F_jE_i = \delta_{ij} \sum_{\lambda \in W\pi}
  [\bil{h_i}{\lambda}]_i 1_\lambda$ holds in the algebra $\T$, for any
  $i,j\in I$.
\end{lem}

\begin{proof}
From the defining relation \ref{def:UU}(d) and the preceding lemma we
have
\begin{align*}
E_i F_j - F_j E_i &= \delta_{ij} \frac{(\sum_{\lambda\in W\pi}
v^{\bil{h_i}{\lambda}} 1_\lambda)^{d_i} - (\sum_{\lambda\in W\pi}
v^{-\bil{h_i}{\lambda}} 1_\lambda)^{d_i} } { v_i - v_i^{-1} } \\
& = \delta_{ij} \frac{(\sum_{\lambda\in W\pi} v^{d_i\bil{h_i}{\lambda}}
1_\lambda) - (\sum_{\lambda\in W\pi} v^{-d_i\bil{h_i}{\lambda}}
1_\lambda) } { v_i - v_i^{-1} } \\
& =\delta_{ij}  \sum_{\lambda\in W\pi} \frac{v^{d_i\bil{h_i}{\lambda}}
 - v^{-d_i\bil{h_i}{\lambda}} } { v_i - v_i^{-1} } 1_\lambda \\
& = \delta_{ij} \sum_{\lambda \in W\pi}
  [\bil{h_i}{\lambda}]_i 1_\lambda
\end{align*}
as desired.
\end{proof}

\begin{lem} \label{tech}
Let $f(K_{H_1},\dots,K_{H_n})$ be an element of the polynomial ring
$\Q(v)[K_{H_1},\dots,K_{H_n}] \subset \UU^0$ and let $\lambda\in
X$. Regard $f$ as a regular function on $\Aff_{\Q(v)}^n$.  The value
of $f(v^{\pm\bil{H_1}{\alpha_j}}K_{H_1}, \dots, v^{\pm
  \bil{H_n}{\alpha_j}} K_{H_n})$ on the point $p_{\lambda}$ is the
same as the value of $f(K_{H_1},\dots,K_{H_n})$ on the point
$p_{\lambda\pm \alpha_j}$, for any $j\in I$.
\end{lem}

\begin{proof}
For any $\lambda \in X$ the value of $f(K_{H_1},\ldots,K_{H_n})$
at the point $p_{\lambda}$ is
\[
f(v^{\bil{H_1}{\lambda}},\ldots,v^{\bil{H_n}{\lambda}}).
\]
Similarly, the value of $f(v^{\pm\bil{H_1}{\alpha_j}}K_{H_1},
\dots, v^{\pm \bil{H_n}{\alpha_j}}K_{H_n})$ at $p_{\lambda}$ is
\[
f(v^{\bil{H_1}{\lambda\pm \alpha_j}}, \ldots,
v^{\bil{H_n}{\lambda\pm \alpha_j}}),
\]
which is the same as the value of $f(K_{H_1},\dots,K_{H_n})$ at
the point $p_{\lambda \pm \alpha_j}$, as desired.
\end{proof}

\begin{lem}\label{john:corollary}
Suppose $\lambda \in W\pi$. Then in the algebra $\T^0$
\pt{a}
$1_{\lambda}(v^{\pm\bil{H_1}{\alpha_j}} K_{H_1}, \ldots,v^{\pm
\bil{H_n}{\alpha_j}} K_{H_n}) = 1_{\lambda +\alpha_j}$ if
$\lambda + \alpha_j \in W\pi$.
\pt{b}
$1_{\lambda}(v^{\pm\bil{H_1}{\alpha_j}} K_{H_1}, \ldots, v^{\pm
\bil{H_n}{\alpha_j}} K_{H_n}) = 0$ if $\lambda+\alpha_j
\notin W\pi$.
\end{lem}

\begin{proof}
Let $D = W\pi \cup \{\omega \pm \alpha_j\mid \omega\in
W\pi,j=1,\ldots,n\}$. Set $Q=\UU/I'(P_D)$. Note that $\T^0$ is a
non-unital subalgebra of $Q$. By Lemma \ref{tech} we have in the
algebra $Q$
\begin{align*}
1_{\lambda}(v^{-\bil{H_1}{\alpha_j}} K_{H_1}, \dots, v^{-
\bil{H_n}{\alpha_j}} K_{H_n})(p_\mu) &= 1_{\lambda}
(p_{\mu-\alpha_j}) \\
&= \delta_{\lambda,\mu-\alpha_j} = \delta_{\lambda+\alpha_j,\mu}
\end{align*}
for any $\mu \in X$, $j \in I$.  The result now follows from the
remarks at the beginning of paragraph \ref{rmks}, since
$\lambda+\al_j\in D$.
\end{proof}

\begin{lem} \label{commutation}
Let $\lambda \in W\pi$. In the algebra $\T$ we have

\pt{a} $E_i 1_{\lambda}=
\begin{cases}
 1_{\lambda+\alpha_i} E_i &\text{if $\lambda+\al_i \in W\pi$}\\
 0 & \text{otherwise}.
\end{cases}$

\pt{b} $F_i 1_{\lambda}=
\begin{cases}
 1_{\lambda-\alpha_i} F_i &\text{if $\lambda-\al_i \in W\pi$}\\
 0 & \text{otherwise}.
\end{cases}$

\pt{c} $1_{\lambda} E_i=
\begin{cases}
 E_i 1_{\lambda-\alpha_i}&\text{if $\lambda-\al_i \in W\pi$}\\
 0 & \text{otherwise}.
\end{cases}$

\pt{d} $1_{\lambda} F_i =
\begin{cases}
 F_i 1_{\lambda+\alpha_i} &\text{if $\lambda+\al_i\in W\pi$}\\
 0 & \text{otherwise}.\end{cases}$
\end{lem}

\begin{proof}
From relation \ref{def:UU}(b) we have $E_iK_h = v^{-\bil{h}{\al_i}}
K_hE_i$ for all $h\in Y$, $i\in I$. For any polynomial
$f(K_{H_1},\ldots,K_{H_n})$ it follows that
\[
E_i\,f(K_{H_1},\ldots,K_{H_n}) = f(v^{-\bil{H_1}{\al_i}}
K_{H_1},\ldots, v^{-\bil{H_n}{\al_i}}K_{H_n}) \,E_i.
\]
Thus in the algebra $\T$ we have
\[
E_i\,1_{\lambda}(K_{H_1},\ldots,K_{H_n})=
1_{\lambda}(v^{-\bil{H_1}{\al_i}}K_{H_1}, \ldots,
v^{-\bil{H_n}{\al_i}}K_{H_n}) \,E_i
\]
which by \ref{john:corollary} equals $1_{\lambda+\al_i}$ if
$\lambda+\al_i\in W\pi$, and equals $0$ if $\lambda+\al_i \notin
W\pi$. This proves part (a); the argument in the remaining cases is
similar.
\end{proof}

Our main result is the following.

\begin{thm} \label{thm:main}
Let $\{H_1, \dots, H_n\}$ be a $\Z$-basis of $Y$. The defining ideal
  $\ideal$ (see \ref{gsa}) of the generalized $q$-Schur algebra
  $\Sch(\pi)$ is the two-sided ideal of $\UU$ generated by
  $I(P_{W\pi})$. In particular, $\ideal$ is generated by its
  intersection with $\UU^0$.
\end{thm}

\begin{proof}
The existence of a quotient map $\T \to \Sch(\pi)$ in \ref{qm} shows
that $\dim \T \ge \dim \Sch(\pi)$.

On the other hand, $\T$ is generated by the elements $E_i$, $F_i$
($i\in I$) and $1_\lambda$ ($\lambda \in W\pi$). The orthogonality of
the idempotents $1_\lambda$ shows that these generators satisfy
\ref{def:Spi}(a), and \ref{lem:d} shows that they satisfy
\ref{def:Spi}(d). They satisfy \ref{def:Spi}(b), (c) by
\ref{commutation} and \ref{def:Spi}(e), (f) are automatic by
\ref{def:UU}(e), (f).  It thus follows that $\T$ is a quotient of
$\Sch(\pi)$, and thus $\dim \T \le \dim \Sch(\pi)$. This proves that
\[
\dim \Sch(\pi) = \dim \T,
\]
and thus the quotient map $\T \to \Sch(\pi)$ is an isomorphism. The
result now follows from the definitions $\T = \UU/J$, $\Sch(\pi) =
\UU/\ideal$, and the definition of $J$ given in \ref{qm}.
\end{proof}

Finally, we wish to obtain a more explicit description of the defining
ideal $I(P_{W\pi})$. We continue to work with an arbitrary basis
$H_1,\ldots, H_n$ of the free abelian group $Y$. Let $L_1,\ldots, L_n$
be the corresponding dual basis of $X$ under the perfect pairing.  We
continue to view the elements $K_a:= K_{H_a}$ ($a=1, \dots, n$) as
coordinate functions on $n$-dimensional affine space
$\Aff_{\Q(v)}^n$. Any element $\lambda\in X$ determines a point
\[
p_\lambda = (v^{\bil{H_1}{\lambda}},\ldots ,v^{\bil{H_n}{\lambda}})\in
\Aff_{\Q(v)}^n
\]
and we regard elements of $\UU^0$ as well-defined functions on $P_X$
via $K_h(p_\lambda) = v^{\bil{h}{\lambda}}$, for any $h\in Y$,
$\lambda \in X$. In the following, we will sometimes use the
convenient shorthand $\lambda_a := \bil{H_a}{\lambda}$, for $a=1,
\dots, n$.

For $h \in Y$, we write $h=\sum c_aH_a\in Y$ in the form $h=h_++h_-$
where $h_+$ (respectively $h_-$) is the part of the sum indexed by all $a$
such that $c_a>0$ (respectively $c_a<0)$. We thus have
$K_h=K_{h_+}K_{h_-}$. Let $\Delta$ be the ideal of $\UU^0$ generated
by all
\[
F_h = \prod_{\lambda \in W\pi}(K_h-v^{\bil{h}{\lambda}}) \qquad (h \in
Y).
\]
Since $K_{h_-} \in \UU^0$ is invertible, an alternate set of
generators for $\Delta$ is given by $G_h = \prod_{\lambda \in
  W\pi}(K_{h_+} - v^{\bil{h}{\lambda}} K_{h_-}^{-1})$ (for $h\in
Y$). We have $K_{h_-} G_h = F_h$ for each $h \in Y$.  The generators
$G_h$ lie in the polynomial ring $\Q(v)[K_1,\ldots,K_n].$

\begin{lem}\label{lem:VDelta}
Consider the set $Z(\Delta)$ consisting of all
$p_\lambda \in P_X$ such that $F(p_\lambda) = 0$ for all $F \in
\Delta$. Then $Z(\Delta) = P_{W\pi}$.
\end{lem}

\begin{proof}
It is clear that $P_{W\pi}\subset Z(\Delta)$. Now take $p\in
Z(\Delta)$. It is clear that $p=p_s$ for some $s=\sum s_aL_a\in
X$. Since every $F_h(p_s)=0$, it follows that for all $h\in Y$ there
exists $\lambda \in W\pi$ with $\langle h,\lambda\rangle=\langle
h,s\rangle$.  If $h=\sum c_aH_a$, then $\langle
h,\lambda\rangle=\langle h,s\rangle$ implies that the point
$(c_1,\ldots,c_n)\in \Aff_{\Q}^n$ lies on the codiminsion one
hyperplane $\mathcal{S}_{\lambda}$ determined by the equation $\sum
(s_a-\lambda_a)H_a=0$.  So if $(c_1,\ldots,c_n)\notin
\bigcup_{\lambda\in W\pi}\mathcal{S}_{\lambda}$, then $\langle
h,\lambda\rangle\neq \langle h,s\rangle$ for all $\lambda \in
W\pi$. But $p_s\in Z(\Delta)$ and so no such point $(c_1,\ldots, c_n)$
can exist.  Thus we must have $\bigcup_{\lambda\in
  W\pi}\mathcal{S}_{\lambda}=\Aff_{\Q}^n$, which implies that
$s=\lambda$ for some $\lambda \in W\pi$. Hence $p_s\in P_{W\pi}$ and
the lemma is proved.
\end{proof}

Let $\Delta'$ be the ideal of $\Q(v)[K_1,\ldots,K_n]$ generated by all
$G_h$. By the comments above, $Z(\Delta)=Z(\Delta')$ and so by the
Lemma we have that $Z(\Delta')=P_{W\pi}$. It will be established next
that $\Q(v)[K_1,\ldots,K_n]/\Delta'$ is reduced, i.e.,\ has no
nilpotent elements. We may work over the algebraic closure $\F$ of
$\Q(v)$ since nilpotent elements remain nilpotent under field
extension.

Let $\mathcal O_{\lambda}$ be the local ring of $\Aff_{\F}^n$ at
$p_{\lambda}$.  Since $Z(\Delta')$ is finite, we have an isomorphism
(see \cite[Chapter 2, \S9, Prop.~6]{Fulton})
\[
\F[K_1,\ldots, K_n]/\Delta'\rightarrow \prod_{\lambda\in W\pi}\mathcal
O_{\lambda}/\Delta'\mathcal O_{\lambda}.
\]
Now if each $\mathcal O_{\lambda}/\Delta'\mathcal O_{\lambda}$ has no
nilpotent elements, then the desired result will be attained. This
will be the case if each $\mathcal O_{\lambda}/\Delta'\mathcal
O_{\lambda}$ is regular since a regular local ring is a domain.

\begin{lem}\label{lem:reg}
The local ring $\mathcal O_{\lambda}/\Delta'\mathcal
O_{\lambda}$ is regular.
\end{lem}

\begin{proof}
Fix $\lambda \in W\pi$ and choose $C=(c_{ab})\in \SL_n(\mathbb Z)$
such that for all $a,b$, $c_{ab}\geq 0$ and for all $\mu\neq \lambda$,
$\bil{H_a'}{\lambda} \neq \bil{H_a'}{\mu}$ where $H_a' = c_{a1}H_1 +
\ldots + c_{an}H_{n}$. (The existence of such a $C$ is proved below.)
The condition $c_{ab}\geq 0$ for all $a,b$ implies that $F_{H_a'} =
G_{H_a'}$ and the condition that $\bil{H_a'}{\lambda} \neq
\bil{H_a'}{\mu}$ for $\mu \neq \lambda$ implies that
$K_{H_a'}(p_{\lambda}) \neq K_{H_a'}(p_{\mu})$ for all $\lambda \neq
\mu$. Hence, the generator $F_{H_a'}$ of $\Delta'$ is an invertible
multiple of $(K_{H_a'}-v^{\bil{H_a'}{\lambda}})$ in the local ring
$\mathcal O_{\lambda}$.  A simple calculation shows that
\[
\frac {\partial( K_{H_a'}-v^{\bil{H_a'}{\lambda}}) }{\partial
  K_b}(p_{\lambda}) = c_{ab}v^{-\lambda_b}\,v^{\bil{H_a'}{\lambda}}.
\]
Hence the Jacobian matrix for
$K_{H_1'}-v^{\bil{H_1'}{\lambda}},\ldots,
K_{H_n'}-v^{\bil{H_n'}{\lambda}}$ at $p_{\lambda}$ is the matrix
\[
\mathcal J=(c_{ab}v^{-\lambda_b}\,v^{\bil{H_a'}{\lambda}})_{1\le a,b
  \le n}.
\]
If for all $a$ and $b$ we multiply column $b$ of $\mathcal J$ by
$v^{\lambda_b}$ and we divide row $a$ by $v^{\bil{H_a'}{\lambda}}$, we
obtain the matrix $C$. Since $C\in \SL_n(\mathbb Z)$, it follows that
$\det (\mathcal J)\neq 0$. Therefore $\mathcal
O_{\lambda}/\Delta'\mathcal O_{\lambda}$ is a regular local ring.

It remains only to establish the existence of a suitable $C$ with the
desired properties.  This amounts to choosing a ``generic'' element of
$\SL_n(\Z)$ with non-negative entries. Here are the details. Fix
$\lambda \in W\pi$. If $\bil{H_a'}{\lambda} = \bil{H_a'}{\mu}$, then
$\mathbf{c}(a) := (c_{a1},\ldots, c_{an})$ lies on the hyperplane
$\mathcal{S}_{\mu}= (\lambda_1-\mu_1)H_1+\cdots
(\lambda_n-\mu_n)H_n=0$. Thus we need to find a $C$ such that none of
its rows lies in the finite union $\bigcup_{\mu\neq \lambda}
\mathcal{S}_{\mu}$.

Let $C$ be any matrix in $\SL_n(\mathbb Z)$ with non-negative
entries. Suppose some $\mathbf{c}(a)$ lies on one of the $\mathcal
S_{\mu}$.  Pick a row $\mathbf{c}(b)\notin \mathcal S_{\mu}$; this is
possible since $C$ has rank $n$. If $\mathbf{c}(a)$ is not in some of
the hyperplanes, choose $m>0$ so that $\mathbf{c}(a)+m\mathbf{c}(b)$
is also not in those hyperplanes.  Then $\mathbf{c}(a)+m\mathbf{c}(b)$
is not in $\mathcal{S}_\mu$. Replace $\mathbf{c}(a)$ by
$\mathbf{c}(a)+m\mathbf{c}(b)$; the determinant does not change and
the entries of $C$ remain non-negative.  Repeating that process, we
replace $\mathbf{c}(a)$ with a row that is in none of the hyperplanes.
\end{proof}

We can now obtain the desired explicit description of $I(P_{W\pi})$.

\begin{prop}
  $I(P_{W\pi})$ is the ideal of $\Q(v)[K_{H_1}, \dots, K_{H_n}]$
  generated by all $\prod_{\lambda \in W\pi} (K_h -
  v^{\bil{h}{\lambda}})$ for $h \in Y$.
\end{prop}

\begin{proof}
It suffices to show that $I(P_{W\pi}) = \Delta'$. By the remarks
following the proof of Lemma \ref{lem:VDelta}, we have $P_{W\pi} =
Z(\Delta')$. By Lemma \ref{lem:reg} it follows that $\Delta'$ is a
radical ideal. Hence by the Nullstellensatz we have $I(P_{W\pi}) =
I(Z(\Delta')) = \sqrt{\Delta'} = \Delta'$, as desired.
\end{proof}

\begin{cor} \label{cor:main}
  The defining ideal $\Lambda$ of $\Sch(\pi)$ is the two sided ideal
  of $\UU$ generated by all $\prod_{\lambda \in W\pi} (K_h -
  v^{\bil{h}{\lambda}})$ for $h \in Y$.
\end{cor}

\begin{proof}
  This is immediate from Theorem \ref{thm:main} and the preceding
  proposition.
\end{proof}

Although this result gives an infinite set of explicit generators for
the ideal $\Lambda$, it should be noted that $\Lambda$ is actually
finitely generated (since $\UU$ is Noetherian). Picking out a finite
set of generators of the above form seems difficult in general,
although we have managed to do that in various examples, treated below
in Section \ref{sec:ex}.

\section{The classical case}\noindent
There is a $v=1$ analogue of the main result, Theorem \ref{thm:main},
and the purpose of this section is to formulate it.  This reduces the
problem of finding relations to define a classical generalized Schur
algebra $S(\pi)$ as a quotient of $\U$ to the problem of finding
equations that cut out a set of points corresponding to the finite set
$W\pi$ of weights, as an affine variety.

\subsection{}\label{def:Ualt}
As before, we fix an arbitrary basis $\{H_1, \dots, H_n \}$ of the
abelian group $Y$. This is a $\Q$-basis of $\mathfrak{h} =
\Q\otimes_\Z Y \subset \g$.  It is easy to check that relations
\ref{U:classical}(a), (b), (c) for general $h \in \h$ are consequences
of those same relations imposed on just the $h = H_1, \dots, H_n$. (In
fact, this is true if one takes $\{H_1, \dots, H_n \}$ to be any
$\Q$-basis of $\mathfrak{h} = \Q\otimes_\Z Y$.)  It follows that $\U$
is the associative algebra with 1 over $\Q$ generated by the
\[
e_i, f_i\quad (i \in I) \qquad H_1, \dots, H_n
\]
and satisfying the relations

(a) $[H_a, H_b]=0$;

(b) $[H_a, e_i] = \bil{H_a}{\alpha_i} e_i$;

(c) $[H_a, f_i] = -\bil{H_a}{\alpha_i} f_i$;

(d) $[e_i, f_j] = \delta_{ij} h_i$;

(e) $(\text{ad } e_i)^{1 - \bil{h_i}{\alpha_j}} e_j = 0 =  (
\text{ad } f_i)^{1 - \bil{h_i}{\alpha_j}} f_j \quad (i\ne j)$

\noindent
for all $a,b = 1, \dots, n$ and all $i,j \in I$, where $\{H_1, \dots,
H_n\}$ is any $\Q$-basis of $\Q \otimes_\Z Y \subset \g$.

\subsection{}
By the Poincare-Birkhoff-Witt theorem it follows that
\[
\U^0 \simeq \Q[H_1, \dots, H_n];
\]
that is, the zero part of $\U$ is isomorphic with the polynomial
algebra in the $H_1, \dots, H_n$. We regard
$\Q[H_1, \dots, H_n]$ as the affine algebra of the
affine variety $\Aff_\Q^n$, and we regard $H_a$ for $a=1, \dots,
n$ as coordinate functions on $\Aff_\Q^n$.

Any subset $D \subset X$ determines a corresponding finite set of
points
\[
 P_D = \{ ({\bil{H_1}{\lambda}}, \dots,
  {\bil{H_n}{\lambda}}) \mid \lambda \in D \}
\]
in $\Aff_\Q^n$. Let $I(P_D)$ be the vanishing ideal of the finite set $P_D$
in $\Aff_\Q^n$. This is an ideal of $\Q[H_1, \dots,
H_n] \simeq \U^0$.

\begin{thm} \label{thm:classical:main}
Let $\{H_1, \dots, H_n\}$ be any $\Q$-basis of $\Q \otimes_\Z Y
\subset \g$.  The defining ideal $\ideal$ (see \ref{classical:gsa}) of
the generalized Schur algebra $S(\pi)$ is the two-sided ideal of $\U$
generated by $I(P_{W\pi})$. In particular, $\ideal$ is generated by
its intersection with $\U^0$.
\end{thm}

The proof parallels the proof of the quantized case in the preceding
section, and the details are left to the reader.

Finally, we also have the analogue of Corollary \ref{cor:main}, proved
by similar arguments as in the quantized case. Again, we leave the
details to the reader and simply state the result, which gives a more
explicit description of the ideal $\Lambda$ in the classical
situation.

\begin{cor}\label{cor:classical:main}
  The defining ideal $\Lambda$ of $S(\pi)$ is the two sided ideal
  of $\U$ generated by all $\prod_{\lambda \in W\pi} (h -
  {\bil{h}{\lambda}})$ for $h \in Y$.
\end{cor}

\section{Examples}\label{sec:ex} \noindent
We give here as motivation some specific applications of the main
results, Theorems \ref{thm:main} and \ref{thm:classical:main}. Not
only can we now treat the presentations from \cite{DG:PSA} and
\cite{DGS} by a unified approach, we are also able to include a
completely new example, namely the generalized Schur algebra coming
from the set of dominant weights of a tensor power of the spin
representation in type $B$.

Throughout this section, we write $e_{a,b}$ for the matrix unit
$e_{a,b} = (\delta_{ai}\delta_{bj})_{1 \le i,j, \le n}$.  In general,
$Y$ may be regarded as a sublattice of a Cartan subalgebra $\h$ of the
Lie algebra $\g$ and hence $X$ may be identified with a sublattice of
$\h^*$ (see \ref{U:classical}).

\subsection{Type $A_{n-1}$.}\label{ex:A:classical}

Let $\U = \U(\gl_n)$ be the universal enveloping algebra of the
general linear Lie algebra $\gl_n$. We take $X = Y = \Z^n$ with bases
$\{H_1, \dots, H_n \}$ of $Y$ and $\{\varepsilon_1, \dots,
\varepsilon_n \}$ of $X$, with pairing $\bil{H_a}{\varepsilon_b} =
\delta_{ab}$ for all pairs of indices $a,b$. The root datum is given
by $h_i = H_i - H_{i+1}$ and $\alpha_i = \varepsilon_i -
\varepsilon_{i+1}$, for all $i \in I = \{1, \dots, n-1\}$. The Cartan
datum is defined by $(i,j) = \delta_{ij}$ for all $i,j \in I$.  We may
identify $H_a$ with the matrix unit $e_{a,a}$ in $\gl_n$, for $a = 1,
\dots, n$.

Let $\pi$ be the set of dominant weights occurring in the $r$th tensor
power $V^{\otimes r}$ of the vector representation $V$.  The
corresponding set of points in $\Aff_\Q^n$ is precisely the set of
partitions of $r$ into $n$ parts (with 0 allowed), so $W\pi$
corresponds with the set of $n$-part compositions of $r$. The
vanishing ideal on this discrete set of points is generated by
\[
 (\textstyle\sum_a H_a) - r, \qquad
 H_a(H_a-1)\cdots (H_a-r)
\]
for $a=1, \dots, n$.  In the quantized case, this is replaced by
\[
 (\textstyle\prod_a K_{H_a}) - v^r, \qquad
 (K_{H_a}-1)(K_{H_a}-v)\cdots (K_{H_a}-v^r)
\]
for $a=1, \dots, n$. Combining these observations with the
presentation of $\U$ and $\UU$ given in \ref{def:Ualt} and
\ref{def:UUalt} respectively, we recover the main results of
\cite{DG:PSA}, obtaining generators and relations for the classical
Schur algebras and their $q$-analogues.

\subsection{Type $A_{n-1}$.} \label{ex:RSA}
We retain the root datum and notation of \ref{ex:A:classical}, except
now we take $\pi$ to be the set of dominant weights of the module
$V^{\otimes r} \otimes {V^*}^{\otimes s}$.  The example treated in
\ref{ex:A:classical} is just the special case $s = 0$. For general
$r,s$ the generalized Schur algebra of type $A_{n-1}$ determined by
$\pi$ is called a rational Schur algebra \cite{DD}.

The set $W\pi$ in this case corresponds to the set of points in
$\Aff_\Q^n$ whose components sum to $r-s$, with each proper partial
sum of components lying in the interval $[-s,r]$. (In case $s=0$ this
is just a composition of $r$ into $n$ parts.) The vanishing ideal of
this discrete point set is generated by
\[
 (\textstyle\sum_a H_a) - r + s, \qquad
 (P+s)(P+s-1)\cdots (P-r)
\]
for each proper partial sum $P$ of $H_a$ ($a =1, \dots, n$). Combined
with the presentation of $\U$ given in \ref{def:Ualt}, this gives a
presentation of the rational Schur algebra $S(n; r,s)$ studied in
\cite{DD}. A different (and simpler) presentation of $S(n;r,s)$ is
given in \cite[7.3]{DD}, but in that presentation the quotient map $\U
\to S(n;r,s)$ is not the natural one sending generators onto
generators.

\subsection{Type $B_n$}\label{ex:B_n}
In type $B_n$ we take $\U=\U(\so_{2n+1}(\Q))$, where the Lie algebra
$\so_{2n+1}(\Q)$ is defined relative to the bilinear form whose
matrix (in the standard basis) is
\[
\begin{pmatrix} 0&I_n&0\\I_n&0&0\\0&0&1\end{pmatrix}.
\]
The $e_{i,i} - e_{n+i,n+i}$ ($1 \le i \le n$) form a basis for
the diagonal Cartan subalgebra $\h$ of $\g$.

Take $X' = Y' = \Aff_\Q^n$ with fixed $\Q$-bases $\{ \varepsilon_i: 1 \le i
\le n \}$ of $X'$ and $\{ H_i: 1 \le i \le n \}$ of $Y'$. We define a
bilinear pairing $\bil{\ }{\ }:Y' \times X' \to \Q$ by the rule
$\bil{H_i}{\varepsilon_j} = \delta_{ij}$.  We take $X=\Z^n \cup
((\frac{1}{2}, \dots, \frac{1}{2}) + \Z^n) \subset X'$, with
$\Z$-basis given by
\[
\{ \varepsilon_i: 1 \le i \le n-1 \} \cup \{ \tfrac{1}{2}(\varepsilon_1
+ \cdots + \varepsilon_n)\} .
\]
The simple roots are given by
\[
\alpha_i = \varepsilon_i - \varepsilon_{i+1}\ \ (1 \le i \le n-1);
\qquad \alpha_n = \varepsilon_n
\]
and the simple coroots are
\[
h_i = H_i - H_{i+1}\ \ (1 \le i \le n-1);\qquad
h_n = 2H_n.
\]
The fundamental weights in this case are given by
\begin{equation*}
\varpi_i = \varepsilon_1 + \cdots + \varepsilon_i \quad (1\le i \le n-1),
\qquad \varpi_{n} = (\varepsilon_1 + \cdots + \varepsilon_{n})/2
\end{equation*}
and the set $\{ \varpi_1, \dots, \varpi_n \}$ forms another $\Z$-basis
of $X$.  We take $Y$ to be the $\Z$-span of the simple coroots $h_1,
\dots, h_n$ in $Y'$; it is easy to see that the restriction of the
pairing $\bil{\ }{\ }:Y' \times X' \to \Q$ gives a perfect pairing
$\bil{\ }{\ }:Y \times X \to \Z$ since $\bil{h_i}{\varpi_j} =
\delta_{ij}$ for all $1 \le i,j \le n$.  Thus the datum
$(X,\{\alpha_i\}, Y, \{h_i\})$ is a root datum of type $B_n$.  The
indexing set $I$ is $\{1, \dots, n\}$ and we may identify the $H_i$
with the elements $e_{i,i} - e_{n+i,n+i} \in \g$, for $i\in I$.

Let $\pi$ be the set $\Pi^+(V^{\otimes r})$, the set of dominant
weights occurring in a weight space decomposition of $V^{\otimes r}$,
where $V$ is the vector representation. Then $\pi$ is a saturated
subset of $X^+$ and the set $W\pi$ may be identified with the set of
all signed $n$-part compositions of $r-j$ for $0 \le j \le r$ (see
\cite[Proposition 1.3.1]{DGS}). By a signed $n$-part composition of
$r$ we mean a tuple $(\lambda_1, \dots, \lambda_n) \in \Z^n$ such that $\sum
|\lambda_i| = r$.  The vanishing ideal of this discrete point set in
$\Aff_\Q^n$ is generated by the elements
\begin{gather*}
(H_i + r)(H_i+r-1)(H_i+r-2)\cdots(H_i - r) \\
(J+r)(J+r-1)(J+r-2)\cdots(J-r+1)(J-r)
\end{gather*}
where $i \in I$ and $J = \pm H_1 \pm H_2 \pm \cdots \pm H_n$ varies
over all the $2^n$ possible sign choices. Combined with the
presentation of $\U$ given in \ref{def:Ualt}, these observations
recover the presentation of $S(\pi)$ given in \cite[Theorem
2.1.1]{DGS}.

We note that the module $V^{\otimes r}$ is {\em not} in general
saturated in type $B$; see \cite[Remark 1.3.4]{DGS}.

To treat the $q$-analogue $\Sch(\pi)$ one would need to first express
each $H_i$ in terms of the integral basis $h_1, \dots, h_n$ of $Y$ and
then rewrite the above elements in terms of the $h_i$ with
denominators cleared. This leads to relations which are rather
unpleasantly non uniform, but which are easily quantizable.
Another approach is to use the set of relations coming from
Corollary \ref{cor:main}.

\subsection{Type $D_n$}\label{ex:D_n}
In type $D_n$ we take $\U=\U(\so_{2n}(\Q))$, where the Lie algebra
$\so_{2n}(\Q)$ is defined relative to the bilinear form whose matrix
(in the standard basis) is
\[
\begin{pmatrix} 0&I_n\\I_n&0\end{pmatrix}.
\]
The $e_{i,i} - e_{n+i,n+i}$ ($1 \le i \le n$) form a basis for
the diagonal Cartan subalgebra $\h$ of $\g$.

Take $X'$, $Y'$ the same as defined in \ref{ex:B_n}, with the same
pairing $\bil{\ }{\ }: Y' \times X' \to \Q$ and the fixed $\Q$-bases
$\{ \varepsilon_i : 1\le i \le n \}$, $\{ H_i : 1\le i \le n \}$ of
$Y'$, $X'$ respectively such that $\bil{H_i}{\varepsilon_j} =
\delta_{ij}$.

We take $X=\Z^n \cup ((\frac{1}{2},\dots,\frac{1}{2})+\Z^n)$ just the
same as in \ref{ex:B_n}, with the same $\Z$-basis $\{ \varepsilon_i: 1
\le i \le n-1 \} \cup \{ \frac{1}{2}(\varepsilon_1 + \cdots +
\varepsilon_n)\}$.  The simple roots are in this case given by
\[
\alpha_i = \varepsilon_i - \varepsilon_{i+1}\ \ (1 \le i \le n-1);
\qquad \alpha_n = \varepsilon_{n-1}+\varepsilon_n
\]
and the simple coroots are given by
\[
h_i = H_i - H_{i+1}\ \ (1 \le i \le n-1);\qquad
h_n = H_{n-1}+H_n.
\]
The fundamental weights are in this case given by
\begin{equation*}
\begin{gathered}
\varpi_i = \varepsilon_1 + \cdots + \varepsilon_i \quad (1\le i \le
n-2),\\ \varpi_{n-1} = (\varepsilon_1 + \cdots +\varepsilon_{n-1} -
\varepsilon_{n})/2, \quad \varpi_{n} = (\varepsilon_1 + \cdots +
\varepsilon_{n})/2.
\end{gathered}
\end{equation*}
and the set $\{ \varpi_1, \dots, \varpi_n \}$ is a $\Z$-basis of $X$.
We define $Y$ to be the $\Z$-span of the simple coroots $h_1, \dots,
h_n$.  The indexing set $I$ is $\{1, \dots, n\}$ and we may identify
the $H_i$ with the elements $e_{i,i} - e_{n+i,n+i} \in \g$, for $i\in
I$. Since $\bil{h_i}{\varpi_j} = \delta_{ij}$ for all $1 \le i,j \le
n$ one sees immediately that $(X, \{\alpha_i\}, Y, \{h_i\})$ is a root
datum of type $D_n$.

Let $\pi$ be the set $\Pi^+(V^{\otimes r})$, the set of dominant
weights occurring in a weight space decomposition of $V^{\otimes r}$,
where $V$ is the vector representation. Then $\pi$ is a saturated
subset of $X^+$ and the set $W\pi$ may be identified with the set of
all signed $n$-part compositions of $r-2j$ for $0 \le j \le [r/2]$
(see \cite[Proposition 1.3.1]{DGS}).  The vanishing ideal of this
discrete point set in $\Aff_\Q^n$ is generated by the elements
\begin{gather*}
(J+r)(J+r-2)(J+r-4)\cdots(J-r+2)(J-r)
\end{gather*}
where $i \in I$ and $J = \pm H_1 \pm H_2 \pm \cdots \pm H_n$ varies
over all the $2^n$ possible sign choices.  Combined with the
presentation of $\U$ given in \ref{def:Ualt}, these observations
recover the presentation of $S(\pi)$ given in \cite[Theorem
2.3.1]{DGS}.

We note that the module $V^{\otimes r}$ is saturated in type $D$; see
\cite[Proposition 1.3.3]{DGS}.

To treat the $q$-analogue $\Sch(\pi)$ one faces precisely the same
difficulty as discussed at the end of \ref{ex:B_n}. Again one needs to
express each $H_i$ in terms of the integral basis $h_1, \dots, h_n$ of
$Y$ and then rewrite the above elements in terms of the $h_i$ with
denominators cleared.  As in type $B$, another approach would be to
appeal to Corollary \ref{cor:main}.

\subsection{Type $C_n$}\label{ex:C_n}
In type $C_n$ we take $\U=\U(\sp_{2n}(\Q))$, where the Lie algebra
$\sp_{2n}(\Q)$ is defined relative to the bilinear form whose
matrix (in the standard basis) is
\[
\begin{pmatrix} 0&-I_n\\-I_n&0\end{pmatrix}.
\]
The $e_{i,i} - e_{n+i,n+i}$ ($1 \le i \le n$) form a basis for
the diagonal Cartan subalgebra $\h$ of $\g$.

We take $Y=\Z^n$ and $X=\Z^n$, and we choose bases $\{ H_i: 1 \le i
\le n \}$ of $Y$ and $\{ \varepsilon_i: 1 \le i \le n \}$ of $X$, with
pairing given by $\bil{H_i}{\varepsilon_j} = \delta_{ij}$. The simple
roots are the
\[
\alpha_i = \varepsilon_i - \varepsilon_{i+1}\ \ (1 \le i \le n-1);
\qquad \alpha_n = 2\varepsilon_n
\]
and the simple coroots are
\[
h_i = H_i - H_{i+1}\ \ (1 \le i \le n-1);\qquad
h_n = H_n.
\]
The indexing set is $I = \{1, \dots, n\}$, and we may identify the
$H_i$ with the elements $e_{i,i} - e_{n+i,n+i}$, for $i\in I$.

Let $\pi$ be the set $\Pi^+(V^{\otimes r})$, the set of dominant
weights occurring in a weight space decomposition of $V^{\otimes r}$,
where $V$ is the vector representation. Then $\pi$ is a saturated
subset of $X^+$ and the set $W\pi$ may be identified with the set of
all signed $n$-part compositions of $r-2j$ for $0 \le j \le [r/2]$
(see \cite[Proposition 1.3.1]{DGS}).  The vanishing ideal of this
discrete point set in $\Aff_\Q^n$ is generated by the elements
\begin{gather*}
(J+r)(J+r-2)(J+r-4)\cdots(J-r+2)(J-r)
\end{gather*}
where $J = \pm H_1 \pm H_2 \pm \cdots \pm H_n$ varies over all the
$2^n$ possible sign choices.  Combined with the presentation of $\U$
given in \ref{def:Ualt}, these observations recover the presentation
of $S(\pi)$ given in \cite[Theorem 2.2.1]{DGS}.

We note that the module $V^{\otimes r}$ is saturated in type $C$; see
\cite[Proposition 1.3.3]{DGS}.

In this case it is easy to treat the $q$-analogue, since the set $\{
H_1, \dots, H_n \}$ is a $\Z$-basis of $Y$. So the defining ideal
$\Lambda$ of $\Sch(\pi)$ is generated by the elements
\begin{gather*}
(K_J-v^r)(K_J-v^{r-2})(K_J-v^{r-4})\cdots(K_J-v^{-r+2})(K_J-v^{-r})
\end{gather*}
where $J = \pm H_1 \pm H_2 \pm \cdots \pm H_n$ varies over all the
$2^n$ possible sign choices.

\subsection{Type $B_n$ --- spin module}\label{ex:B_n:Spin}
We treat an entirely new example in this subsection. Retain the root
datum as defined in \ref{ex:B_n}, except now we take $\pi$ to be the
set $\Pi^+(S^{\otimes r})$, the set of dominant weights occurring in a
weight space decomposition of $S^{\otimes r}$, where $S$ is the spin
representation of $\so_{2n+1}(\Q)$.  We note that the module
$S^{\otimes r}$ is saturated in type $B$; see Appendix~\ref{app:spin}.

The set $W\pi$ in case $r=1$ is the set $\Pi(S)$ of weights of $S$;
the $\lambda$ in this set are precisely the elements of $X$ satisfying
the condition $\bil{H_i}{\lambda} = \pm \frac{1}{2}$ for all $i \in
I$. (The Weyl group $W$ acts through signed permutations.)

For general $r$ the description of $W\pi$ divides naturally into two
cases, depending on the parity of $r$. If $r = 2m$ is even, then one
sees easily by a simple induction that the set $W\pi$ consists of
those $\lambda \in X$ such that $\bil{H_i}{\lambda} \in \{ 0, \pm 1,
\dots, \pm m \}$ for each $i \in I$. If $r = 2m+1$ then $W\pi$
consists of those $\lambda \in X$ such that $\bil{H_i}{\lambda} \in
\pm\frac{1}{2} + \{ 0, \pm 1, \dots, \pm m \}$ for each $i \in I$.

The vanishing ideal of this discrete point set in $\Aff_\Q^n$ is generated
by the elements
\[
(H_i + \tfrac{r}{2})(H_i+\tfrac{r}{2}-1)\cdots(H_i - \tfrac{r}{2}+1)(H_i
- \tfrac{r}{2})
\]
where $i \in I$.  Combined with the presentation of $\U$ given in
\ref{def:Ualt}, this gives the following description. The generalized
Schur algebra $S(\pi)$ is the associative algebra with 1 over $\Q$
generated by
\[
e_i, f_i, H_i \quad (i = 1, \dots, n)
\]
and satisfying the relations

(a) $H_i H_j = H_j H_i$;

(b) $H_i e_j - e_j H_i = \bil{H_i}{\alpha_j} e_j$;

(c) $H_i f_j - f_j H_i = -\bil{H_i}{\alpha_j} f_j$;

(d) $e_i f_j - f_j e_i = \delta_{ij} h_i$;

(e) $\displaystyle \sum_{s+s'=1-\bil{h_i}{\alpha_j}} (-1)^{s'}
\divided{e_i}{s}e_j \divided{e_i}{s'} = 0$ ($i \ne j$);

(f) $\displaystyle \sum_{s+s'=1-\bil{h_i}{\alpha_j}} (-1)^{s'}
\divided{f_i}{s}f_j \divided{f_i}{s'} = 0$ ($i \ne j$);

(g) $(H_i + \frac{r}{2})(H_i+\frac{r}{2}-1)\cdots(H_i - \frac{r}{2}+1)(H_i
- \frac{r}{2}) = 0$

\noindent
for all $i,j \in I = \{1, \dots, n\}$. Here $\divided{f_i}{s} =
f_i^s/(s!)$, $\divided{e_i}{s} = e_i^s/(s!)$ are the usual divided
powers.

Treating the $q$-analogue $\Sch(\pi)$ one faces difficulties already
discussed in previous examples, and probably the best one can do is to
appeal to Corollary \ref{cor:main}.

\appendix
\section{Spin representation in type $B$} \label{app:spin}\noindent
The purpose of this appendix is to prove that the tensor powers of the
spin representation in type $B$ are always saturated modules, in the
sense defined in the introduction.

Retain the root datum introduced in \ref{ex:B_n}.  This is the root
datum for the simply-connected covering group $\mathsf{Spin}(2n+1)$.
The fundamental weights for the system are $\varpi_1=\varepsilon_1$,
$\varpi_2 = \varepsilon_1 + \varepsilon_2, \dots, \varpi_{n-1} =
\varepsilon_1 + \cdots + \varepsilon_{n-1}$, $\varpi_n = \frac{1}{2}
(\varepsilon_1 + \cdots + \varepsilon_n)$. A basis for the set of
dominant weights for the special orthogonal group $\mathsf{SO}(2n+1)$
is $\varpi_1$, $\varpi_2, \dots, \varpi_{n-1}$, $2\varpi_n$; i.e.,
the dominant weights for the irreducible representations of
$\mathsf{SO}(2n+1)$ are the non-negative integral combinations of those
weights.

For each non-negative integral combination $\omega = t_1\varpi_1 +
\cdots + t_{n-1}\varpi_{n-1} + t_n2\varpi_n $, let $|\omega| =
\sum_{i=1}^nt_i$.

\begin{lem} \label{lem:C:1}
For each simple root $\alpha_i$, $|\omega -\alpha_i|\le |\omega|$.
\end{lem}

\begin{proof} Suppose that $\omega =m_1\varpi_1+\cdots
+m_{n-1}\varpi_{n-1}+m_n2\varpi_n$, and that $\omega
-\alpha_j=k_1\varpi_1+\cdots +k_{n-1}\varpi_{n-1}+k_n2\varpi_n$.
We need to show that $\sum_{j=1}^n k_j\le m =\sum_{j=1}^n m_j$.

We have
\[
m=m_1 +\cdots + m_n=\bil{h_1}{\omega }+\cdots \bil{h_{n-1}}{\omega }+
\frac{\bil{h_n}{\omega }}{2}
\]
and
\begin{align*}
k_1 + \cdots + k_n &= \bil{ h_1}{\omega -\alpha_i} + \cdots \bil{
h_{n-1}}{\omega-\alpha_i}+ \frac{\bil{ h_n}{\omega -\alpha_i}}{2} \\ &=
m-\bil{h_1}{(\alpha_i }+\cdots \bil{h_{n-1}}{\alpha_i}+
\frac{\bil{h_n}{\alpha_i } }{2}.
\end{align*}

Using the Cartan datum for $B_n$, we see that $\bil{h_1}{\alpha_1 } +
\cdots \bil{h_{n-1}}{\alpha_1 }+ \frac{\bil{h_n}{\alpha_1 }}{2}=1$ and
that $\bil{h_1}{\alpha_i }+ \cdots + \bil{h_{n-1}}{\alpha_i}+
\frac{\bil{h_n}{\alpha_i }}{2}=0$, for $i>1$.  Hence, $k_1+\dots
+k_n=m $ or $k_1+\dots +k_n=m-1$.
\end{proof}

For each non-negative integer $m$, consider the set $W_m$ of dominant
weights (for $\mathsf{SO}(2n+1)$) of the form $\omega
=m_1\varpi_1+\cdots +m_{n-1}\varpi_{n-1}+m_n2\varpi_n$ with
$\sum_{j=1}^nm_j=m$ (and each $m_j\ge 0$). Those are the dominant
weights $\omega $ for $\mathsf{SO}(2n+1)$ such that $|\omega |=m$.

\begin{lem}\label{lem:C:2}
For each non-negative integer m, the set of dominant weights
$\bigcup_{j=0}^mW_j$ is saturated.
\end{lem}

\begin{proof} Let $\omega\in W_m$. Then $\omega =m_1\varpi_1+\cdots
+m_{n-1}\varpi_{n-1}+m_n2\varpi_n$, where each $m_i\ge 0$, and
$m=\sum_{j=1}^nm_j=|\omega |$. Suppose that
$\omega-\sum_{j=1}^nt_j\alpha_j$ is a dominant weight where each
$t_i$ is a non-negative integer. We need to show that
$\omega-\sum_{j=1}^nt_j\alpha_j$ is an element of some $W_k$ with
$0\le k\le m$, i.e., that $0\le |\omega-\sum_{j=1}^nt_j\alpha_j|\le
m$. Because $\omega-\sum_{j=1}^nt_j\alpha_j$ is dominant, $0\le
|\omega-\sum_{j=1}^nt_j\alpha_j|$, and by Lemma \ref{lem:C:1} we
have $|\omega-\sum_{j=1}^nt_j\alpha_j|\le m$.
\end{proof}

Let $S$ be the irreducible spin representation of highest weight
$\varpi_n$, and let $V$ be the vector representation of
$\mathsf{SO}(2n+1)$, of highest weight $\varpi_1$. Let $S^{\otimes k}$
be the $k$th tensor power of $S$ and let $\wedge^kV$ be the $k$th
exterior power of $V$. Note that $S^{\otimes 2}$ may be regarded as a
module for $\mathsf{SO}(2n+1)$; i.e., the representation of
$\mathsf{Spin}(2n+1)$ on $S\otimes S$ factors through
$\mathsf{SO}(2n+1)$.

\begin{lem} \label{lem:C:3}
$S\otimes S$ is the sum $k + V + \wedge^2V+ \cdots
+ \wedge^{n-1}V + \wedge^nV$ of irreducible modules (for
$\mathsf{SO}(2n+1)$) of highest weights $0$, $\varpi_1$, $\varpi_2,
\dots, \varpi_{n-1}$, and $2\varpi_n$.
\end{lem}

\begin{proof} This is well known.  \end{proof}

\begin{prop} \label{prop:C:4}
The set of highest weights of the irreducible factors of $S^{\otimes
2m}$ is $\bigcup_{j=0}^mW_j$.
\end{prop}

\begin{proof}
The highest weights $\omega$ of the irreducible modules in $S\otimes
S$ lie in $W_0\cup W_1$, by Lemma \ref{lem:C:3}. We claim that the
highest weight $\omega $ of an irreducible module in the even tensor
power $S^{\otimes 2m}$ lie in $\bigcup_{j=0}^mW_j$. In fact, as we
observed just above, that is true for $m=1$. Assume that it is true
for $S^{\otimes (2m-2)}$; i.e., that $\omega\in \bigcup_{j=0}^{m-1}W_j$
for the highest weight $\omega $ of each irreducible submodule of
$S^{\otimes (2m-2)}$. The irreducible submodules of $S^{\otimes 2m}$
occur as submodules of an irreducible factor $M$ of $S\otimes S$ and
an irreducible factor $N$ of $S^{\otimes (2m-2)}$. By the case $m=1$
and by the inductive hypothesis, if $\omega $ equals the highest
weight of $M\otimes N$, then $\omega\in \bigcup_{j=0}^mW_j$. Any other
dominant weight of $M\otimes N$ has the form $\omega
-\sum_{j=1}^nd_i\alpha_i $, for non-negative integers $d_i$. By
Lemma \ref{lem:C:2}, $\omega -\sum_{j=1}^nd_i\alpha_i $ lies in
$\bigcup_{j=0}^mW_j$. Hence, the irreducible submodules of $S^{\otimes
2m}$ have highest weights in $\bigcup_{j=0}^mW_j$.

We show next that all dominant weights $\omega $ for
$\mathsf{SO}(2n+1)$ which lie in $\bigcup_{j=0}^mW_j$ occur as the
highest weights of irreducible submodules of $S^{\otimes 2m}$. Lemma \ref{lem:C:2} 
establishes that result for $m=1$. Suppose that the
result holds for $m-1$, with $m\ge 2$. Then each weight $\omega $ in
$\bigcup_{j=0}^{m-1}W_j$ occurs as the highest weight of an irreducible
submodule of $S^{\otimes (2m-2)}$, and because $S^{\otimes (2m-2)}$
occurs as a submodule of $S^{\otimes 2m} = S\otimes S \otimes
S^{\otimes (2m-2)}$ (since $k$ occurs as a submodule of $S\otimes S$),
each weight $\omega $ in $\bigcup_{j=0}^{m-1}W_j$ occurs as the highest
weight of an irreducible submodule of $S^{\otimes 2m}$ also.  Next
take any weight $\omega $ in $W_m$. Write $\omega $ as either $\delta
+\varpi_j$ for some $j<m$ or as $\omega =\delta +2\varpi_n$, for
some dominant weight $\delta $. We have $\delta\in W_{m-1}$, and so
$\delta $ is the highest weight of some irreducible submodule $N$ of
$S^{\otimes (2m-2)}$, and $\varpi_j$ (or $2\varpi_n$ as the case may
be) is the highest weight of some irreducible submodule $M$ of
$S\otimes S$. Hence, $\omega $ is the highest weight of $N\otimes M$
of $S^{\otimes 2m}$, and so it is the highest weight of an irreducible
submodule of the submodule $N\otimes M$ of $S^{\otimes 2m}$.
\end{proof}

For weights of the form $\varpi_n+\omega $, where $\omega
=t_1\varpi_1+\cdots +t_{n-1}\varpi_{n-1}+t_n2\varpi_n$, for integers
$t_i$, let $|\varpi_n+\omega |$ equal the value $|\omega|$.

\begin{lem} \label{lem:C:5}
The set of dominant weights of the form $\varpi_n +\omega $, for
$\omega\in \bigcup_{j=0}^mW_j$, is a saturated set.
\end{lem}
\medskip
The proof is the same as that of Lemma \ref{lem:C:2}.
\bigskip

\begin{prop} \label{prop:C:6}
The set of highest weights of the irreducible factors of $S^{\otimes
(2m+1)}$ is the set $\{ \varpi_n +\omega: \omega \in \bigcup_{j=0}^mW_j
\}$.
\end{prop}

\begin{proof} The module $S^{\otimes (2m+1)} = S\otimes S^{\otimes 2m}$
is the sum of submodules of the form $S\otimes M$ where $M$ is an
irreducible module with highest weight in $\bigcup_{j=0}^mW_j$. The
highest weight of such a tensor product is the highest weight of an
irreducible submodule of $S\otimes M$. Such highest weights range over
the set described in Lemma \ref{lem:C:5}, by the proof of Proposition
\ref{prop:C:4}. To show that that set exhausts the highest weights of
irreducible submodules of $S^{\otimes (2m+1)}$, we have that the
highest weight of $S\otimes M$ has the form $\varpi_n+\omega $ with
$\omega\in\bigcup_{j=0}^mW_j$; hence, the dominant weights in $S\otimes
M$ have the form
\[ \textstyle
\varpi_n+\omega -\sum_{j=1}^nd_i\alpha_i = \textstyle \varpi_n+\omega
-\sum_{j=1}^nd_i\alpha_i ,
\]
with $d_i\ge 0$, where the dominant weight $\omega
-\sum_{j=1}^nd_i\alpha_i$ is an element of $\bigcup_{j=0}^mW_j$, by
Lemma \ref{lem:C:2}.
\end{proof}

\begin{thm}
  In type $B_n$, any tensor power $S^{\otimes r}$ of the spin module
  $S$ in type $B_n$ is a saturated module (in the sense defined in the
  introduction).
\end{thm}

\begin{proof}
  This follows from the definition of saturated modules, given in the
  introduction, along with results \ref{lem:C:2}, \ref{prop:C:4},
  \ref{lem:C:5}, and \ref{prop:C:6}.
\end{proof}


\end{document}